\documentclass[12pt,twoside]{article}
\usepackage{a4wide}
\usepackage[ngerman, english]{babel}
\usepackage[ansinew]{inputenc}
\usepackage{authblk}
\usepackage{mathpazo}
\usepackage{setspace}
\usepackage{amsmath}
\usepackage{amsmath,mathtools}
\usepackage{amsthm}
\usepackage{amsfonts}
\usepackage{amssymb}
\usepackage[nottoc]{tocbibind} 
\usepackage{fancyhdr} 
\usepackage{graphicx}
\usepackage{tikz}
\usepackage{bm}
\usepackage{pgfplots} 
\usepackage{enumerate}
\usepackage{verbatim}
\usepackage{todonotes}
\usepackage{caption}
\usepackage{subcaption} 
\usepackage{wrapfig} 

\usetikzlibrary{shapes,backgrounds,calc}

\newtheorem{theorem}{Theorem}

\newtheorem{corollary}[theorem]{Corollary}

\newtheorem{lemma}[theorem]{Lemma}

\numberwithin{equation}{section}
\numberwithin{theorem}{section}

\theoremstyle{definition}

\newtheorem{remark}[theorem]{Remark}

\DeclareMathOperator{\vol}{vol}
\DeclareMathOperator{\res}{res}

\DeclareMathOperator*{\Res}{Res} 
\newcommand{\unity}{{1\!\!\!\:\mathrm{l}}}
\newcommand{\x}{\mathrm{x}}
\newcommand{\y}{\mathrm{y}}
\newcommand{\z}{\mathrm{z}}
\newcommand{\imm}{\mathrm{f}}
\newcommand{\qi}{\mathrm{i}}
\newcommand{\qj}{\mathrm{j}}

\newcommand{\blankpage}{\newpage\thispagestyle{empty}\mbox{}\newpage} 

\title{Solutions of the Sinh-Gordon Equation of Spectral Genus Two and constrained Willmore Tori I}
\author[$\dagger$]{M.\ Knopf}
\affil[$\dagger$]{School of Business Informatics and Mathematics, University of Mannheim, D-68131 Mannheim, Germany}
\author[$\dagger$]{R.\ Pe\~na Hoepner}
\author[$\dagger$,$\ddag$]{M.\ U.\ Schmidt}
\affil[$\ddag$]{email:schmidt@math.uni-mannheim.de, ORCID:0000-0002-1775-4701}

\date{\today}

\begin{document}
\maketitle
\begin{abstract}
In \cite{PS:89} and \cite{Hi} the solutions of the elliptic sinh-Gordon equation of finite spectral genus $g$ are investigated. These solutions are parametrized by complex matrix-valued polynomials called potentials. On the space of these potentials there act two commuting flows. The orbits of these flows are called Polynomial Killing fields. For $g=2$ they are double periodic. The eigenvalues of these matrix-valued polynomials are preserved along the flows and determine the lattice of periods. For $g = 2$, we investigate the level sets of these eigenvalues, which are called isospectral sets, and the dependence of the period lattice on the isospectral sets. The limiting cases of spectral genus one and zero are included. These limiting cases are used to construct on every elliptic curve three conformal maps to $\mathbb{H}$ which are constrained Willmore. Finally, the Willmore functional is calculated in dependence of the conformal class.
\end{abstract}

\pagestyle{fancy}	
\fancyhf{}
\renewcommand{\sectionmark}[1]{\markboth{#1}{}} 
 \fancyhead[EL]{\scshape Solutions of the sinh-Gordon Equation of Spectral Genus Two}
 \fancyhead[OL]{\scshape\thesection.\,\leftmark}
 \fancyfoot[EC,OC]{\thepage}
 \renewcommand{\headrulewidth}{0.5pt}
 \renewcommand{\footrulewidth}{0.5pt}  
\headsep=35pt 
\voffset=-20pt 
\setcounter{page}{1} 
\section{Introduction}
The elliptic \textit{sinh-Gordon equation} is given by
\begin{equation}
\Delta u + 2\sinh(2u) = 0,
\label{eq_sinh}
\end{equation}
where $\Delta$ is the Laplacian of $\mathbb{R}^2$ with respect to the Euclidean metric and $u:\mathbb{R}^2\to\mathbb{R}$ is a twice partially differentiable real-valued function. Pinkall and Sterling constructed so called finite type solutions in terms of a finite dimensional ODE system~\cite[Section~5]{PS:89}. We represent this system as in~\cite[Section~2]{HKS1} by two commuting flows on a certain space of matrix valued polynomials. These polynomials are called potentials and their degree is related to the genus of a naturally assigned algebraic curve. For spectral genus zero there is only the trivial solution $u=0$. Uwe Abresch investigated spectral genus one solutions and spectral genus two solutions whose spectral curves have an additional symmetry~\cite{Abr}. Here we take up the work \cite{PH} on general spectral genus two solutions with the following space of potentials parametrized by $\alpha$, $\beta$ and $\gamma$:
\begin{align*}
\mathcal{P}_2:=\bigg\{\zeta = 
\begin{pmatrix} 
\alpha\lambda-\bar\alpha\lambda^2 & -\gamma^{-1}+\beta\lambda-\gamma\lambda^2\\
\gamma\lambda-\bar\beta\lambda^2+\gamma^{-1}\lambda^3 & -\alpha\lambda+\bar\alpha\lambda^2
\end{pmatrix} \,\bigg|\,\alpha,\beta\in\mathbb{C},\gamma\in\mathbb{R}^+\bigg\}.
\end{align*}
On this space the Lax equations define the following vector fields:
\begin{align}
\frac{\partial \zeta}{\partial x}=[\zeta,U(\zeta)]\quad\qquad\frac{\partial \zeta}{\partial y}=[\zeta,V(\zeta)]\label{LaxZeta}
\end{align}
with
\begin{align*}
U(\zeta) :=
\begin{pmatrix}
\frac{\alpha-\bar\alpha}{2} &  -\gamma^{-1}\lambda^{-1}-\gamma\\
\gamma+\gamma^{-1}\lambda & \frac{\bar\alpha-\alpha}{2}
\end{pmatrix},
\,\,\,
V(\zeta) :=i
\begin{pmatrix}
\frac{\alpha+\bar\alpha}{2} &  -\gamma^{-1}\lambda^{-1}+\gamma\\
\gamma-\gamma^{-1}\lambda & -\frac{\alpha+\bar\alpha}{2}.
\end{pmatrix}.
\end{align*}
A direct calculation shows that these two vector fields~\eqref{LaxZeta} commute:
\begin{align}\label{eq:commute}
[U(\zeta),V(\zeta)]+\frac{\partial V(\zeta)}{\partial x}-\frac{\partial U(\zeta)}{\partial y}&=0,&\frac{\partial^2 \zeta}{\partial x\partial y}&=\frac{\partial^2 \zeta}{\partial y\partial x}.
\end{align}
For every initial potential $\zeta (0) = \zeta_0\in\mathcal{P}_2$ the integral curves fit together and form a map $(x,y)\mapsto\zeta(x,y)$ on an open neighborhood of $(0,0)$ in $\mathbb{R}^2$ which is called Polynomial Killing field. The corresponding function $u(x,y):= \ln\gamma(x,y)$ solves the $\sinh$-Gordon equation \eqref{eq_sinh}.\vspace{3mm}

We define the polynomials $a\in\mathbb{C}^4[\lambda]$ of fourth degree as
\begin{align}\label{det a}
\det \zeta&=\lambda a(\lambda).
\end{align}
For all such polynomials $a$ the corresponding level sets
\begin{align}\label{eq:isospectral}
I(a)&=\{\zeta\in\mathcal{P}_2\mid \det\zeta=\lambda a(\lambda)\}
\end{align}
are called isospectral sets. A direct calculation utilizing~\eqref{LaxZeta} shows
\begin{align}\label{integrals of motion}
\frac{\partial a}{\partial x}&=0=\frac{\partial a}{\partial y}.
\end{align}
Therefore the Polynomial Killing fields stay in the isospectral set of the initial potential $\zeta_0\in\mathcal{P}_2$.\vspace{3mm}

\noindent{\bf Outline of the Paper.}
In Section~\ref{sec iso} we show that all isospectral sets $I(a)$ are compact and that the Polynomial Killing fields are defined globally on $(x,y)\in\mathbb{R}^2$ for all initial potentials $\zeta\in\mathcal{P}_2$. Consequently, the vector fields~\eqref{LaxZeta} induce an action
\begin{align}\label{eq:group action}
\phi:\mathbb{R}^2\ni(x,y)&\mapsto\phi(x,y),&\phi(x,y):\mathcal{P}_2&\to\mathcal{P}_2&\zeta&\mapsto\phi(x,y)\zeta.
\end{align}
We shall see that the isospectral sets $I(a)$ are invariant with respect to the action~\eqref{eq:group action} and decompose either in one ore two orbits. For polynomials $a$ with only one orbit in $I(a)$ the isotropy groups of this action
\begin{align}\label{eq:isotropy}
\Gamma_{\zeta}&=\{\gamma\in\mathbb{R}^2\mid\phi(\gamma)\zeta=\zeta\}
\end{align}
are the same for all $\zeta\in I(a)$. In Section~\ref{sec periods} we investigate $\Gamma_{\zeta}$ and their dependence on $a$. We define a sub-lattice $\tilde{\Gamma}_a \subset \Gamma_{\zeta}$ and show that the map $T: a \mapsto \tilde{\Gamma}_a$ extends continuously to general $a$. The isospectral sets $I(a)$ of polynomials $a$ with one or two unitary double roots contain the spectral genus one and spectral genus zero solutions, respectively. In Section~\ref{Sec:g1} we calculate the polynomial Killing fields of these solutions with $g \leq 1$ in terms of elliptic functions and show that the corresponding restriction of the map $T: a \mapsto \tilde{\Gamma}_a$ is surjective. In Section~\ref{sec willmore} we show that for a sub-lattice $\hat{\Gamma}_a \subset \tilde{\Gamma}_a$ all requirements from quaternionic function theory~\cite{PP} are met in order to define a conformal map $\imm_a: \mathbb{C}/\hat{\Gamma}_a \to \mathbb{H}$. For $g \leq 1$, a theorem of~\cite{Bo} shows that these maps are constrained Willmore, i.e.\ critical points of the restriction of the Willmore functional
\begin{equation}
W := \int_{\mathbb{C}/\hat{\Gamma}_a} H^2 \,dA
\label{Willmore}
\end{equation}
to the space of all conformal maps $\imm_a:\mathbb{C}/\hat{\Gamma}_a \to \mathbb{H}$. Here, $H$ is the mean curvature vector and $dA$ the induced volume form on $\mathbb{C}/\hat{\Gamma}_a$. It follows from Section~\ref{Sec:g1} that all elliptic curves occur in this family. Finally, we plot the Willmore functional in dependence of the conformal class.
%
%
%
\section{The Isospectral Sets}\label{sec iso}
Since all potentials in $\mathcal{P}_2$ are traceless, their eigenvalues are determined by the polynomials $a$ via~\eqref{det a}. Due to~\eqref{integrals of motion}, the coefficients of $a$ are constant along the integral curves of these vector fields. In the following theorem we investigate the global structure of these integral curves in dependence of the polynomials $a$. We distinguish between integral curves in $I(a)$ of polynomials $a$ in the following sets:
\begin{align*}
\mathcal{M}_2:&= \{a\in\mathbb{C}^4[\lambda]\, |\, \lambda a(\lambda) = \det(\zeta) \text{ for a } \zeta\in\mathcal{P}_2\},\\
\mathcal{M}_2^1 :&= \{a\in\mathcal{M}_2\mid a\text{ has four pairwise distinct simple roots absent }\mathbb{S}^1\},\\
\mathcal{M}_2^2 :&= \{a\in\mathcal{M}_2\mid a\text{ has one double root on }\mathbb{S}^1\text{ and two simple roots absent }\mathbb{S}^1\},\\
\mathcal{M}_2^3 :&=\{a\in\mathcal{M}_2\mid a \text{ has two distinct double roots on }\mathbb{S}^1\},\\
\mathcal{M}_2^4 :&=\{a\in\mathcal{M}_2\mid a \text{ has a fourth-order root on }\mathbb{S}^1\},\\
\mathcal{M}_2^5 :&=\{a\in\mathcal{M}_2\mid a \text{ has two distinct double roots absent }\mathbb{S}^1\}.
\end{align*}
Clearly $\mathcal{M}_2$ is the disjoint union $\mathcal{M}_2^1\cup\mathcal{M}_2^2\cup\mathcal{M}_2^3\cup\mathcal{M}_2^4\cup\mathcal{M}_2^5$. By definition of $\mathcal{P}_2$ the highest and lowest order coefficients of $a$ are equal to one. In terms of $\alpha$, $\beta$ and $\gamma$ the coefficients of $a$ take the form
\begin{gather}
a(\lambda)=\lambda^4+a_1\lambda^3+a_2\lambda^2+\bar{a}_1\lambda+1\quad\mbox{with}\nonumber\\
\begin{aligned}\label{coefficients a}
a_1&=-\bar\alpha^2-\beta\gamma^{-1}-\bar\beta\gamma\in\mathbb{C}\quad\mbox{and}&a_2&=2\alpha\bar\alpha+\beta\bar\beta+\gamma^2+\gamma^{-2}\in\mathbb{R}.
\end{aligned}\end{gather}
Consequently, the map $\zeta\mapsto a$ can be represented as
\begin{align}
\mathbb{C}\times\mathbb{C}\times\mathbb{R}^+&\rightarrow\mathbb{C}\times\mathbb{R}^+,&(\alpha,\beta,\gamma)&\mapsto(a_1,a_2)\label{func_f}
\end{align}
with $a_1$ and $a_2$ as in~\eqref{coefficients a}. For unitary $\lambda$ the matrices $\lambda^{-\frac{3}{2}}\zeta$ are anti-hermitian. As the determinant $\det$ is the square of a norm on such matrices, $\lambda^2a(\lambda)\ge 0$ for such $\lambda$ and $a\in\mathcal{M}_2$. Therefore, 
$$\mathcal{M}_2=\{a\in\mathbb{C}^4[\lambda]\, |\, a(0)=1,\,\lambda^4\overline{a(\bar\lambda^{-1})}=a(\lambda),\,\lambda^{-2}a(\lambda)\ge 0 \text{ for }\lambda\in\mathbb{S}^1\}.$$
Now we state the main theorem of this section:
\begin{theorem}\label{thm orbits}
The vector fields~\eqref{LaxZeta} induce the action~\eqref{eq:group action} of $\mathbb{R}^2$ on $\mathcal{P}_2$. The isospectral sets $I(a)\subset\mathcal{P}_2$ are compact and have the following properties:
\begin{enumerate}
\item For $a\in \mathcal{M}_2^1$ the isospectral sets $I(a)$ are two-dimensional compact submanifolds of $\mathcal{P}_2$ with transitive group action~\eqref{eq:group action}, i.e.
$$I(a)=\{\phi(x,y)\zeta\mid(x,y)\in\mathbb{R}^2\}\quad\text{for any }\zeta\in I(a).$$
\item For $a\in\mathcal{M}_2^2$ the isospectral sets $I(a)$ are one-dimensional compact subsets of $\mathcal{P}_2$ with transitive group action~\eqref{eq:group action}.
\item For $a\in\mathcal{M}_2^3\cup\mathcal{M}_2^4$ the isospectral set $I(a)$ consists of a single fixed point of the group action~\eqref{eq:group action}.
\item For $a\in\mathcal{M}_2^5$ the isospectral set is a union of two disjoint orbits of the group action~\eqref{eq:group action}:
$$I(a)=\underbrace{\{\zeta\in I(a)\mid\zeta|_{\lambda\in\{\lambda_1,\bar{\lambda}_1^{-1}\}}\neq 0\}}_{K_a}\cup\underbrace{\{\zeta\in I(a)\mid\zeta|_{\lambda\in\{\lambda_1,\bar{\lambda}_1^{-1}\}}=0\}}_{L_a}.$$
Here $\lambda_1$ and $\bar{\lambda}_1^{-1}$ are the double roots of $a$, $K_a$ is a two-dimensional non-compact submanifold of $\mathcal{P}_2$ with closure $\bar{K}_a=I(a)$ and $L_a$ consists of a single point.
\end{enumerate}
\end{theorem}
\begin{proof}
For compactness of $I(a)$, the closedness follows from continuity of~\eqref{func_f} and the boundedness of $\alpha$, $\beta$, $\gamma$ and $\gamma^{-1}$ results from the formula for $a_2$~\eqref{coefficients a}. Hence there exists for all $a\in\mathcal{M}_2$ an $\epsilon>0$ such that for all initial $\zeta_0\in I(a)$ the Polynomial Killing field is defined on $(x,y)\in B(0,\epsilon)\subset\mathbb{R}^2$. Since they stay in the isospectral set of the initial $\zeta\in\mathcal{P}_2,$ the domain of the Polynomial Killing fields extends to $(x,y)\in\mathbb{R}^2$. Hence the two commuting local flows induced by~\eqref{LaxZeta} extend to the action~\eqref{eq:group action} of $\mathbb{R}^2$.\vspace{3mm}

1. In order to show that $I(a)$ is a two-dimensional submanifold for $a\in\mathcal{M}_2^1$, we shall prove that the Jacobian of $\mathcal{P}^2\to\mathbb{C}\times\mathbb{R},\;\zeta\mapsto(a_1,a_2)$ has full rank on $I(a)$ and apply the implicit function theorem. Let $a\in\mathcal{M}_2^1$ and consider 
	\begin{align*}
		\zeta = 
		\begin{pmatrix} 
			\alpha\lambda-\bar\alpha\lambda^2 & -\gamma^{-1}+\beta\lambda-\gamma\lambda^2\\
			\gamma\lambda-\bar\beta\lambda^2+\gamma^{-1}\lambda^3 & -\alpha\lambda+\bar\alpha\lambda^2
		\end{pmatrix}
		=
		\begin{pmatrix} 
			A(\lambda) & B(\lambda)\\
			\lambda\,C(\lambda) & -A(\lambda)
		\end{pmatrix}
		\in I(a).
	\end{align*}
The potential $\zeta$ is rootless as a root $\lambda\in\mathbb{C}\setminus\{0\}$ of $\zeta$ would lead to a double root in $a$ which is excluded by $a\in\mathcal{M}_2^1$. We calculate the Jacobian $J_f$ of the map $(\alpha,\bar{\alpha},\beta,\bar{\beta},\gamma)\mapsto(a_1,\bar{a}_1,a_2)$ at $\zeta$:
\begin{align*}
J_f&=\begin{pmatrix}
0 & -2\bar{\alpha} & -\gamma^{-1} & -\gamma & (\beta\gamma^{-2}-\bar{\beta})\\
-2\alpha & 0 & -\gamma & -\gamma^{-1} & (\bar{\beta}\gamma^{-2}-\beta)\\
2\bar{\alpha} & 2\alpha & \bar{\beta} & \beta & 2(\gamma-\gamma^{-3})
\end{pmatrix}.
\end{align*}
First assume $\alpha\neq 0$. Then the Jacobian can be transformed via Gaussian elimination into
	\begin{align*}
	\begin{pmatrix}
		-2\alpha & 0 & -\gamma & -\gamma^{-1} & (\bar\beta\gamma^{-2}-\beta)\\
		0 & -2\bar{\alpha} & -\gamma^{-1} & -\gamma & (\beta\gamma^{-2}-\bar{\beta})\\
		0 & 0 & \theta & \iota & \kappa
	\end{pmatrix}
	\end{align*}
	where
	\begin{align*}
	\theta &=\bar{\beta}-\gamma\alpha^{-1}\bar{\alpha}-\gamma^{-1}\bar{\alpha}^{-1}\alpha
	=-\alpha^{-1}\bar{\alpha}\,\,C(\bar{\alpha}^{-1}\alpha)\\
	\iota &= \beta-\gamma^{-1}\alpha^{-1}\bar{\alpha}-\gamma\bar{\alpha}^{-1}\alpha
	=\alpha^{-1}\bar{\alpha}\,\,B(\bar{\alpha}^{-1}\alpha)\\
	\kappa &= 2(\gamma-\gamma^{-3})+\bar{\alpha}^{-1}\alpha(\beta\gamma^{-2}-\bar{\beta})+\alpha^{-1}\bar{\alpha}(\bar\beta\gamma^{-2}-{\beta}).
	\end{align*}
	If the Jacobian $J_f$ has no full rank, then $\theta=\iota=\kappa=0$. But then $\zeta$ has a root in $\bar{\alpha}^{-1}\alpha$.\vspace{3mm}
	
If $\alpha=0$, $A(\lambda)=0$ and consequently, $\zeta$ is off-diagonal and has a root if and only if the polynomials $B(\lambda)$ and $C(\lambda)$ have a common root. This is the case if and only if the resultant $\res(B,C)$ equals zero. Since $\zeta$ has no root the resultant must be nonzero
$$\res(B,C)\neq 0.$$
We calculate the resultant as the determinant of the Sylvester matrix (compare \cite{B:LinAlg}):
	\begin{align*}
	\res(B,C) &=
	\det
	\begin{pmatrix}
		-\gamma & \beta & -\gamma^{-1} & 0\\
		0 & -\gamma & \beta & -\gamma^{-1}\\
		\gamma^{-1} & -\bar{\beta} & \gamma & 0\\
		0 & \gamma^{-1} & -\bar{\beta} & \gamma
	\end{pmatrix}\nonumber
	&\hspace{-2.5mm}=\beta^2+\bar{\beta}^2-\beta\bar{\beta}(\gamma^2+\gamma^{-2})+\gamma^4+\gamma^{-4}-2.
	\end{align*}
Since $\alpha=0$, the first two columns of $J_f$ disappear and we conclude by showing that the determinant of the remaining $3\times 3$ matrix is nonzero.
	\begin{align*}
	\det
	\begin{pmatrix}
		-\gamma^{-1} & -\gamma & (\beta\gamma^{-2}-\bar\beta)\\
		-\gamma & -\gamma^{-1} & (\bar\beta\gamma^{-2}-\beta)\\
		\bar{\beta} & \beta & 2(\gamma-\gamma^{-3})
	\end{pmatrix}
	&=-2\gamma^{-1}\,\res(B,C)\not=0.
	\end{align*}
This shows that $J_f$ has full rank at rootless $\zeta$ and the assertion is proven.\vspace{3mm}

For transitivity, we first show that the vector fields~\eqref{LaxZeta} span the tangent space of $I(a)$ at rootless $\zeta$. In fact, the only matrix which commutes with a non-trivial traceless two-by-two matrix must be a multiple of the same matrix. Therefore, a linear combination of the vector fields~\eqref{LaxZeta} vanishes if and only if a real linear combination of $U(\zeta)$ and $V(\zeta)$ is the product of $\zeta$ with a meromorphic function $g(\lambda)$. Since $\gamma\ne0,$ the coefficients of $\lambda$ and $\lambda^{-1}$ in $U(\zeta)$ and $V(\zeta)$ are linearly independent over $\mathbb{R}$, respectively. Besides a first order pole at $\lambda=0,$ all other poles of $g$ have to be roots of $\zeta$ and $\lambda^2g(\lambda)$ is bounded for $\lambda\to\infty$. Hence, for rootless $\zeta$ the map $g$ vanishes and the vector fields~\eqref{LaxZeta} span the tangent space of $I(a)$. By the implicit function theorem, for $a\in\mathcal{M}_2^1$ the orbits of~\eqref{eq:group action} are open and closed in $I(a)$, in particular compact. Being continuous images of $\mathbb{R}^2,$ they are connected and consequently, the compact connected components of $I(a)$.\vspace{3mm}

We just proved that for $a\in\mathcal{M}_2^1$ and for all $\zeta\in I(a)$ the map $(x,y)\mapsto\phi(x,y)\zeta$ is an immersion onto a compact connected component of $I(a)$. Therefore it induces a diffeomorphism from $\mathbb{R}^2/\Gamma_{\zeta}$ (see~\eqref{eq:isotropy}) onto this compact connected component. In particular, $\Gamma_{\zeta}$ is a two-dimensional lattice in $\mathbb{R}^2$ and the connected components of $I(a)$ are two-dimensional tori. Now we claim that the function
\begin{align}\label{eq:morse}
(x,y)&\mapsto\gamma(\phi(x,y)\zeta)
\end{align}
is a Morse function on $\mathbb{R}^2/\Gamma_{\zeta}$. Explicit calculation of~\eqref{LaxZeta} shows that the critical points are those $(x,y)\in\mathbb{R}^2$ with off-diagonal $\phi(x,y)\zeta$. Such potentials are uniquely determined by distributing the roots of $a$ into two groups of roots of $B(\lambda)$ and $C(\lambda)$. There are exactly four such off-diagonal $\zeta$ in $I(a)$. Another calculation shows that the Hessian is non-degenerate and~\eqref{eq:morse} is a Morse function on $\mathbb{R}^2/\Gamma_{\zeta}$. The Morse inequalities show that every orbit contains the four critical points, one maximum, one minimum and two saddle points. An explicit calculation of the Hessian at the critical points yields the same assertion. Consequently, any two connected components of $I(a)$ have at least four common elements and are equal. This finishes the proof of part~1.\vspace{3mm}

2. Let $a\in\mathcal{M}_2^2$. For $\zeta\in\mathcal{P}_2$ and for unitary $\lambda$ the product $\lambda^{-\frac{3}{2}}\zeta$ is an anti-hermitian two-by-two matrix. The determinant is on such matrices the square of a norm. Consequently, $\zeta$ vanishes at all unitary roots of $a(\lambda)$. Therefore, all $\zeta\in I(a)$ are products of potentials of less degree with normalized polynomials in $\mathbb{C}[\lambda]$, whose square divides $a$.

We show the statement analogously to part~1 and transfer the arguments from part~1 by first introducing the set of potentials for genus one:
$$\mathcal{P}_1:=\left\{\left.\hat{\zeta}=\begin{pmatrix} 
0 & -\hat{\beta}^{-1}\\
0 & 0
\end{pmatrix}
+
\begin{pmatrix} 
i\hat{\alpha} & -\hat{\beta}\\
\hat{\beta} & -i\hat{\alpha}
\end{pmatrix}
\hat{\lambda}
+
\begin{pmatrix} 
0 & 0\\
\hat{\beta}^{-1} & 0
\end{pmatrix}
\hat{\lambda}^2\right|\,\hat{\alpha}\in\mathbb{R},\hat{\beta}\in\mathbb{R}^+\right\}.$$
For these potentials $\hat{\zeta}$ and for unitary $\hat{\lambda}$, the matrix $\hat{\lambda}^{-1}\hat{\zeta}$ is anti-hermitian. The determinant equals
\begin{align*}
\det\left(\hat{\zeta}\right)=\hat{\lambda}\left(\hat{\lambda}^2+(\hat{\alpha}^2+\hat{\beta}^2+\hat{\beta}^{-2})\hat{\lambda}+1\right)
=:\hat{\lambda} \hat{a}(\hat{\lambda})=\hat{\lambda}(\hat{\lambda}^2+\hat{a}_1\hat{\lambda}+1),
\end{align*}
where $\hat{a}(\hat{\lambda})\in\mathcal{M}_1$ with
\begin{align*}
\mathcal{M}_1:&=\{\hat{a}\in\mathbb{C}^2[\hat{\lambda}]\, |\, \hat{\lambda}\hat{a}(\hat{\lambda})=\det(\hat{\zeta})\text{ for a }\hat{\zeta}\in\mathcal{P}_1\}\\
&=\{\hat{a}\in\mathbb{C}^2[\hat{\lambda}]\,|\,\hat{a}(0)=1,\,\hat{\lambda}^2\overline{\hat{a}(\bar{\hat{\lambda}}^{-1})}=\hat{a}(\hat{\lambda}),\,\hat{\lambda}^{-1}\hat{a}(\hat{\lambda})\ge 0 \text{ for }\hat{\lambda}\in\mathbb{S}^1\}.
\end{align*}
After redefining the correspondent sets and using the same subscription as in $g=2$, we consider the function
\begin{align*}
\mathbb{R}\times\mathbb{R}^+\to \mathbb{R}^+,\,
\begin{pmatrix}
\hat{\alpha}\\
\hat{\beta}
\end{pmatrix}
\mapsto
\hat{a}_1 = \hat{\alpha}^2+\hat{\beta}^2+\hat{\beta}^{-2},
\end{align*} and one can easily verify that the gradient of this functions vanishes if and only if
\begin{align*}
\begin{pmatrix}
2\hat{\alpha}\\
2\hat{\beta}-2\hat{\beta}^{-3}
\end{pmatrix}
=0
&\Leftrightarrow
\hat{\zeta} = \left[
	\begin{pmatrix}
	0 & -1\\
	0 & 0
	\end{pmatrix}
	+
	\begin{pmatrix}
	0 & 0\\
	1 & 0
	\end{pmatrix}
	\hat{\lambda}	
	\right]
	(\hat{\lambda}+1).
\end{align*}
In particular, if $\hat{a} \in \mathcal{M}_1$ has pairwise different roots, then$$\hat{I}(\hat{a}) := \{\hat{\zeta} \in \mathcal{P}_1 \left|\right. \det{\hat{\zeta}} = \hat{\lambda} \hat{a}(\hat{\lambda}) \}$$
is a compact, one-dimensional submanifold of $(\hat{\alpha},\hat{\beta}) \in \mathbb{R}\times\mathbb{R}^+$. The corresponding Polynomial Killing fields $\hat{\zeta}:\mathbb{R}^2\to\mathcal{P}_1,\,(\hat{x},\hat{y})\mapsto \hat{\zeta}(\hat{x},\hat{y})$ solve the Lax equations
\begin{align*}
\frac{\partial \hat{\zeta}}{\partial \hat{x}}=[\hat{\zeta},\hat{U}(\hat{\zeta})]\quad\qquad\frac{\partial \hat{\zeta}}{\partial \hat{y}}=[\hat{\zeta},\hat{V}(\hat{\zeta})]
\end{align*}
with $\hat{\zeta}(0) \in\mathcal{P}_1$ and
\begin{align*}
\hat{U}(\hat{\zeta}) :=
\begin{pmatrix}
i\hat{\alpha} &  -\hat{\beta}^{-1}\hat{\lambda}^{-1}-\hat{\beta}\\
\hat{\beta}+\hat{\beta}^{-1}\hat{\lambda} & -i\hat{\alpha}
\end{pmatrix}\quad
\hat{V}(\hat{\zeta}) :=
i
\begin{pmatrix}
0 &  -\hat{\beta}^{-1}\hat{\lambda}^{-1}+\hat{\beta}\\
\hat{\beta}-\hat{\beta}^{-1}\hat{\lambda} & 0
\end{pmatrix}.
\end{align*}
The entries $\hat{\alpha}:\mathbb{R}^2\to\mathbb{R}$ and $\hat{\beta}:\mathbb{R}^2\to\mathbb{R}^+$ of a Polynomial Killing Field $\hat{\zeta}$ satisfy 
\begin{align*}
\frac{\partial \hat{\alpha}}{\partial \hat{x}}&=0,&\frac{\partial \hat{\alpha}}{\partial \hat{y}} &= 2(\hat{\beta}^{-2}-\hat{\beta}^{2}),&\frac{\partial \hat{\beta}}{\partial \hat{x}}&=0,&\frac{\partial \hat{\beta}}{\partial \hat{y}}&=2\hat{\alpha}\hat{\beta}.
\end{align*}
We now transform these Polynomial Killing fields to the Polynomial Killing fields $\zeta$ with $\zeta(0) \in I(a)$ and $a \in \mathcal{M}_2^2$. Let $e^{2i\varphi}$ be the double root and $\hat{\lambda}_1 e^{-2i\varphi}$, $\hat{\lambda}_1^{-1}e^{-2i\varphi}$ the two simple roots of $a(\lambda)$ with $\hat{\lambda}_1 \in (-1,0)$. The transformation is conducted by
\begin{align*}
\zeta(x, y):=e^{3i\varphi}
\begin{pmatrix}
1 & 0 \\
0 & ie^{i\varphi}
\end{pmatrix}
\hat{\zeta}(\hat{x}, \hat{y})
\begin{pmatrix}
1 & 0 \\
0 & -ie^{-i\varphi}
\end{pmatrix} (\lambda - e^{2i\varphi}).
\end{align*}
Here, the parameters $(\hat{x},\hat{y}, \hat{\lambda})$ of $\hat{\zeta}$ are related to the corresponding parameters $(x,y,\lambda)$ of $\zeta$ via
\begin{align*}
\begin{pmatrix} \hat{x} \\ \hat{y} \end{pmatrix} & = 
\begin{pmatrix} \cos(\varphi) & -\sin(\varphi) \\ \sin(\varphi) & \cos(\varphi) \end{pmatrix}
\begin{pmatrix} x \\ y \end{pmatrix},&\hat{\lambda}&=-e^{2i\varphi}\lambda.
\end{align*}
By an explicit calculation $\hat{\zeta}(0,0)\mapsto\zeta(0,0)$ is a homomorphism from $\hat{I}(\hat{a})$ onto $I(a)$.
The analogously applied arguments from part~1 show that for $\hat{a}(\hat{\lambda}) = (\hat{\lambda} - \hat{\lambda}_1)(\hat{\lambda} - \hat{\lambda}_1^{-1})$ the set $\hat{I}(\hat{a})$ is a compact one-dimensional manifold with the action $(\hat{y},\hat{\zeta})\mapsto\hat{\phi}(\hat{y})\hat{\zeta}$ of $\hat{y}\in\mathbb{R}$. Each connected component of $\hat{I}(\hat{a})$ is diffeomorphic to $\mathbb{S}^1$. The action is transitive and $\hat{I}(\hat{a})$ connected, since the Morse function $\hat{y}\mapsto\hat{\beta}(\hat{y})$ has two critical points in $\hat{I}(\hat{a})$.\vspace{3mm}

3. The elements $a\in\mathcal{M}_2^3\cup\mathcal{M}_2^4$ take the form $a(\lambda)=(\lambda-\lambda_1)^2(\lambda-\bar{\lambda}_1)^2$ with $|\lambda_1|=1$. The same argument as in the beginning of part~2 shows that all potentials $\zeta\in I(a)$ are products of potentials of first degree with $(\lambda-\lambda_1)(\lambda-\bar{\lambda}_1)$. By definition of $\mathcal{P}_2$ such $\zeta$ are off-diagonal with $\gamma^2=\lambda_1\bar{\lambda}_1=1$. Hence, there is only one such element of the form
$$\zeta=\left[
	\begin{pmatrix}
	0 & -1\\
	0 & 0
	\end{pmatrix}
	+
	\begin{pmatrix}
	0 & 0\\
	1 & 0
	\end{pmatrix}
	\lambda	
	\right]
	(\lambda-\lambda_1)(\lambda-\bar \lambda_1).$$
Since $I(a)$ is invariant with respect to~\eqref{eq:group action}, the vector fields~\eqref{LaxZeta} vanish on $I(a)$.\vspace{3mm}

4. The elements $a\in\mathcal{M}_2^5$ take the form $(\lambda-\lambda_1)^2(\lambda-\bar{\lambda}_1^{-1})^2$ with $|\lambda_1|\ne1$ and $\lambda_1^2\bar{\lambda}^{-2}_1=1$. The condition $\lambda^{-2}a(\lambda)\ge0$ for $\lambda\in\mathbb{S}^1$ implies $\lambda_1\in\mathbb{R}\setminus\{-1,0,1\}$. For $\zeta\in L_a$ one can argue analogously to part~3 and conclude that $L_a$ contains only the following potential
	\begin{align*}
	  \zeta=\left[
	\begin{pmatrix}
	  0 & -1\\
	  0 & 0
	\end{pmatrix}
	+
	\begin{pmatrix}
	  0 & 0\\
	  1 & 0
	\end{pmatrix}
	  \lambda	
	  \right]
	  (\lambda-\lambda_1)(\lambda-\lambda_1^{-1}).
	\end{align*}
Due to the form of~\eqref{LaxZeta}, the roots of $\zeta$ are preserved along~\eqref{eq:group action} and $L_a$ is an orbit.\vspace{3mm}

For $K_a$ the submanifold property as well as its dimension follows from the corresponding arguments in case~1. These arguments show also that the orbits of $\zeta\in K_a$ are open and closed in $K_a$ and furthermore, the orbits have to be the 
connected components of $K_a$. By lack of compactness of these orbits we consider the closures of these orbits rather than the orbits themselves. Since $I(a)$ is compact, the closures are compact and either equal to the orbits or to the union with $L_a$. The closure of each orbit contains a maximum and a minimum of the function~\eqref{eq:morse} which can either be a local extremum in $K_a$ or the only element of $L_a$. In particular, these elements are off-diagonal potentials. In the present case $I(a)$ contains three such off-diagonal potentials. One is of course the trivial solution of~\eqref{LaxZeta} in $L_a$. A direct calculation shows that \eqref{eq:morse} takes on $L_a$ the value $\gamma = 1$ and at the other two elements in $K_a$ values larger or smaller than $\gamma = 1$. In order to proceed, we now study the subset of potentials with $\gamma = 1$:
\begin{eqnarray*}
a(\lambda) & = & (\lambda - \lambda_1)^2(\lambda - \lambda_1^{-1})^2 \\
& = & \lambda^4 - 2(\lambda_1 + \lambda_1^{-1})\lambda^3 + (\lambda_1^{-2} + 4 + \lambda_1^2)\lambda^2 -2(\lambda_1 + \lambda_1^{-1})\lambda + 1.
\end{eqnarray*}
Due to \eqref{coefficients a} we obtain
\begin{align*}
-2(\lambda_1 + \lambda_1^{-1})=& -\bar{\alpha}^2 - \beta - \bar{\beta},&
\lambda_1^2 + 2 + \lambda_1^{-2} & = 2 \alpha\bar{\alpha} + \beta\bar{\beta}.
\end{align*}
The second equation implies $|\Re(\beta)|\le|\lambda_1+\lambda_1^{-1}|$ and the
first $\frac{\bar{\alpha}^2}{2}=\lambda_1+\lambda_1^{-1}-\Re(\beta)$. Hence,
$\alpha$ is imaginary for $\lambda_1<0$ and real for $\lambda_1>0$. We set
$\alpha=\sqrt{2}i\x$ and $\beta=-\z+i\y$ for $\lambda_1<0$ and
$\alpha=\sqrt{2}\x$ and $\beta=\z+i\y$ for $\lambda_1>0$. In both cases we
obtain
\begin{align*}
|\lambda_1+\lambda_1^{-1}|&=\x^2+\z,&|\lambda_1+\lambda_1^{-1}|^2&=4\x^2+\y^2+\z^2
&\mbox{with }(\x,\y,\z)\in\mathbb{R}^3.
\end{align*}
We eliminate $z$ by the first equation and arrive at:
\begin{align*}\x^4-(|\lambda_1+\lambda_1^{-1}|-2)\x^2+\y^2&=0,&
\Longleftrightarrow&&y&=\pm\x\sqrt{|\lambda_1+\lambda_1^{-1}|-2-\x^2}.
\end{align*}
Since $|\lambda_1+\lambda_1^{-1}|>2$ the $(\x,\y)$-graph looks like $\infty$
with an ordinary double point at $(\x,\y)=(0,0)$ which corresponds to the
single element of $L_a$. All other points are no critical
points of $\gamma$ and on the corresponding orbits $\gamma$ takes values
smaller and larger than  $1$. Consequently, the orbit of any $\zeta\in K_a$
with $\gamma=1$ contains both off-diagonal elements of $K_a$. This implies
first that any orbit of $K_a$ containes an offdiagonal element and second
that $K_a$ is a single orbit, connected, non-compact and dense in $I(a)$.
\end{proof}
\section{Lattices of Periods}\label{sec periods}
Due to Theorem~\ref{thm orbits}, the isospectral set $I(a)$ of each $a\in\mathcal{M}_2\setminus\mathcal{M}_2^5$ contains exactly one orbit and $\Gamma_{\zeta}$ from~\eqref{eq:isotropy} does not depend on $\zeta\in I(a)$. We identify $(x,y)\in\mathbb{R}^2$ with $x+i y\in\mathbb{C}$ and define
\begin{align}\label{lattice a}
\Gamma_a:=\{x+i y\in\mathbb{C}\mid\forall\zeta\in I(a):\phi(x,y)(\zeta)=\zeta\}.
\end{align}
This group $\Gamma_a$ is an abelian and normal subgroup of $\mathbb{C}$ with quotient group $\mathbb{C}/\Gamma_a$. For $a\in\mathcal{M}_2^1,$ Theorem~\ref{thm orbits} proves that $\Gamma_a$ is a discrete subgroup with compact quotient and therefore, a lattice
$$\Gamma_a=\omega_1\mathbb{Z}+\omega_2\mathbb{Z}$$
with (over $\mathbb{R}$) linearly independent  generators $\omega_1$, $\omega_2\in\mathbb{C}$. The choice of these generators is not unique. Given two lattices $\Gamma,\Gamma'\subset\mathbb{C}$ we call them {\it isomorphic} if they originate from one another through a rotation-dilation. In~\cite[Chapter~VI.1]{FB} it is proven, that each lattice $\Gamma$ in $\mathbb{C}$ is up to a rotation-dilation isomorphic to $\Gamma_\tau:=\mathbb{Z}+\mathbb{Z}\tau$ with
\begin{align}\label{fundamental}
\tau&\in\{\tau\in\mathbb{C}\mid\Im(\tau)>0,\,|\Re(\tau)|\le\tfrac{1}{2},\,\|\tau\|\ge 1\}.
\end{align}
Furthermore, the corresponding $\tau$ is unique up to the identifications of the following $\tau$:
\begin{align*}
-\frac{1}{2}+iy\sim\frac{1}{2}+iy&\mbox{ for }y\in[\tfrac{\sqrt{3}}{2},\infty),&-x+i\sqrt{1-x}\sim x+i\sqrt{1-x}&\mbox{ for }x\in[0,\tfrac{1}{2}].
\end{align*}
Let $\mathcal{F}$ denote the space of such $\tau$ with the quotient topology of the subset~\eqref{fundamental} of $\mathbb{C}$ divided by the relation $\sim$. Hence, there exists a unique map
\begin{align*}
T:\mathcal{M}_2^1&\to\mathcal{F},&a&\mapsto\tau_a,
\end{align*}
such that $\Gamma_a$ is isomorphic to $\Gamma_{\tau_a}$. We now want to prove that $T$ has a unique surjective continuous extension to $\mathcal{M}_2^2\cup\mathcal{M}_2^3$. In order to investigate the dependence of $\tau_a$ on $a\in\mathcal{M}_2^1$ we introduce for any $\omega\in\Gamma_a$ the monodromy $M_\omega$ with eigenvalues $\mu_\omega$. Let $\zeta:\mathbb{R}^2\to\mathcal{P}_2$ be a Polynomial Killing field with initial potential $\zeta_0\in\mathcal{P}_2$. The fundamental solution $F$ of the following system of ordinary differential equations is called frame:
\begin{align}
	\frac{\partial F}{\partial x}&=F U(\zeta),&\frac{\partial F}{\partial y}&=F V(\zeta),&F(0,0)&=\mathbb{1}.\label{ODE_frame}
\end{align}
The first equation in~\eqref{eq:commute} is called Maurer-Cartan equation and ensures that $F$ is indeed a function of $(x,y)\in\mathbb{R}^2$. For $a\in\mathcal{M}_2^3\cup\mathcal{M}_2^4$ both $U(\zeta)$ and $V(\zeta)$ are constant and $F$ is equal to $\exp(xU(\zeta)+yV(\zeta))$. For $a\in\mathcal{M}_2^2$ $F$ is calculated in~\cite{KSS}. We identify $(x,y)\in\mathbb{R}^2$ with $z=x+iy\in\mathbb{C}$ and consider $F$ as a function on $z\in\mathbb{C}$. For all $\omega\in\Gamma_a$ the value $M_\omega$ of $F$ at $\omega$ is called monodromy. Since $F^{-1}\zeta_0F$ solve~\eqref{LaxZeta} we have
$$\zeta=F^{-1}\zeta_0F.$$
In particular, for all $\omega\in\Gamma_a$ the monodromy $M_\omega$ commutes with $\zeta_0$ and maps the eigenspaces of $\zeta_0$ into themselves. Since $\zeta_0$ is traceless these eigenspaces can only have more than one dimension at the roots of $\zeta_0$. For $\zeta_0\in I(a)$ with $a\in\mathcal{M}_2^1$ they do not exist and the one-dimensional eigenspaces of each $\zeta_0$ are parametrized by the smooth Riemann surface
\begin{align}\label{eq:spectral curve}
\Sigma^\ast=\{(\lambda,\nu)\in(\mathbb{C}\setminus\{0\})\times\mathbb{C}\mid\det(\nu\unity-\zeta_0)=\nu^2+\lambda a(\lambda)=0\}.
\end{align}
The identity $\bar{\lambda}^4a(\bar{\lambda}^{-1})=\overline{a(\lambda)}$ induces the second of the following involutions:
\begin{align*}
\sigma:(\lambda,\nu)&\mapsto(\lambda,-\nu),&
\rho:(\lambda,\nu)&\mapsto(\bar{\lambda}^{-1},-\bar{\lambda}^{-3}\bar{\nu}).
\end{align*}
The fixed points of $\rho$ are the elements of $\Sigma^\ast$~\eqref{eq:spectral curve} with $\lambda\in\mathbb{S}^1$ since $\lambda^{-2}a(\lambda)\ge 0$ for such $\lambda$. The monodromies $M_\omega$ act on the one-dimensional eigenspaces of $\zeta_0$ as the multiplication with a function $\mu_\omega:\Sigma^\ast\to\mathbb{C}\setminus\{0\}$. The involutions act on $\mu_\omega$ as
\begin{align}\label{involutions mu}
\sigma^\ast\mu_\omega&=\mu_\omega^{-1},&\rho^\ast\mu_\omega&=\bar{\mu}^{-1}_\omega.
\end{align}
Our investigation of the map $a\mapsto \Gamma_a$ is based on the following description of $\Gamma_a$:
\begin{lemma}\label{characterisation 1}
For all $a\in\mathcal{M}_2^1$ the elements of $\Gamma_a$ are characterized as those $\omega\in\mathbb{C}$ such that the function $\exp(\omega\lambda^{-1}\nu)$ on $\Sigma^\ast$ factorizes into the product of a holomorphic function $\mu_\omega$ on $\Sigma^\ast$ obeying~\eqref{involutions mu} with a holomorphic function on $\Sigma^\ast$, which extends holomorphically to $\lambda=0$ and takes there the value $1$.
\end{lemma}
\begin{proof}
Let $G$ denote the group of holomorphic maps $\mathbb{C}\setminus\{0\}\to\mathrm{SL}(2,\mathbb{C})$ with the two subgroups:
\begin{itemize}
\item $G_-$: The holomorphic maps $\mathbb{C}\!\setminus\!\{0\}\!\to\!\mathrm{SL}(2,\mathbb{C})$ which take on $\mathbb{S}^1$ values in $\mathrm{SU}(2,\mathbb{C})$,
\item $G_+$: The holomorphic maps $\mathbb{C}\to\mathrm{SL}(2,\mathbb{C})$ which take at $0\in\mathbb{C}$ values in $\mathrm{SL}\!^+(2,\mathbb{C})$.
\end{itemize}
Here, $\mathrm{SL}\!^+(2,\mathbb{C})$ denotes the subgroup of $\mathrm{SL}(2,\mathbb{C})$ consisting of upper triangular matrices with real and positive diagonal entries. The corresponding Lie algebras are denoted by $\mathfrak{g}$, the holomorphic maps $\mathbb{C}\setminus\{0\}\to\mathrm{sl}(2,\mathbb{C})$ with
\begin{itemize}
\item $\mathfrak{g}_-$: The holomorphic maps $\mathbb{C}\setminus\{0\}\to\mathrm{sl}(2,\mathbb{C})$ which take on $\mathbb{S}^1$ values in $\mathrm{su}(2,\mathbb{C})$,
\item $\mathfrak{g}_+$: The holomorphic maps $\mathbb{C}\to\mathrm{sl}(2,\mathbb{C})$ which take at $0\in\mathbb{C}$ values in $\mathrm{sl}\!^+(2,\mathbb{C})$.
\end{itemize}
Here, $\mathrm{sl}\!^+(2,\mathbb{C})$ denotes the traceless upper triangular matrices with real diagonal entries, which is the Lie algebra of $\mathrm{SL}\!^+(2,\mathbb{C})$. For all $\lambda\in\mathbb{C}\setminus\{0\}$ the elements $g$ of $G_-$ and $\gamma$ of $\mathfrak{g}_-$ obey
\begin{align*}
g(\lambda)&=\overline{(g^T)^{-1}(\bar{\lambda}^{-1})},& \gamma(\lambda)&=-\overline{\gamma^T(\bar{\lambda}^{-1})}.
\end{align*}
Therefore, the elements of $G_-\cap G_+$ and $\mathfrak{g}_-\cap\mathfrak{g}_+$ extend to holomorphic maps on $\mathbb{P}^1$, which are constant. This implies $G_-\cap G_+=\{\unity\}$ and $\mathfrak{g}_-\cap\mathfrak{g}_+=\{0\}$. For $\zeta_0\in\mathcal{P}_2$ let $z\mapsto F(z)$ denote the corresponding frame~\eqref{ODE_frame}. We calculate
\begin{multline*}
\tfrac{d}{dz}\left(F^{-1}(z)\exp(z\lambda^{-1}\zeta_0)\right)\exp(-z\lambda^{-1}\zeta_0)F(z)=\\
=-U(F^{-1}(z)\zeta_0F(z))dx-V(F^{-1}(z)\zeta_0F(z))dy+F^{-1}(z)\lambda^{-1}\zeta_0F(z)dz.
\end{multline*}
The 1-form on the right hand side takes values in $\mathfrak{g}_+$ because the 1-form
$$\lambda^{-1}\zeta dz-U(\zeta)dx-V(\zeta)dy=\begin{pmatrix}\frac{\alpha}{2}dz+\frac{\bar{\alpha}}{2}d\bar{z}-\bar{\alpha}\lambda dz&\beta dz+\gamma d\bar{z}-\gamma\lambda dz\\-\bar{\beta}\lambda dz-\gamma^{-1}\lambda d\bar{z}+\gamma^{-1}\lambda^2 dz&-\frac{\alpha}{2}dz-\frac{\bar{\alpha}}{2}d\bar{z}+\bar{\alpha}\lambda dz\end{pmatrix}.$$
takes values in $\mathfrak{g}_+$ at all $\zeta\in\mathcal{P}_2$. Therefore $\lambda\mapsto F^{-1}(z)\exp(z\lambda^{-1}\zeta_0)$ belongs to $G_+$ and $\lambda\mapsto F(z)$ is the unique element of $G_-$, whose product with an element of $G_+$ is equal to $\exp(z\lambda^{-1}\zeta_0)$. An element $\omega\in\mathbb{C}$ belongs to $\Gamma_a$, if and only if $F(\omega)$ and $\zeta_0$ commute. Equivalently, both factors $(g_-,g_+)\in G_-\times G_+$ of $\exp(\omega\lambda^{-1}\zeta_0) = g_- g_+$ commute with $\zeta_0$. For $\zeta_0$ without roots on $\lambda\in\mathbb{C}\setminus\{0\}$ this is equivalent to both factors taking the form
\begin{align*}
&\unity f(\lambda)+g(\lambda)\zeta&\text{with holomorphic}&&f,g&:\mathbb{C}\setminus\{0\}\to\mathbb{C}.
\end{align*}
They act on the eigenspaces of $\zeta_0$ as the multiplication with the function $f(\lambda)+g(\lambda)\nu$ on $\Sigma^\ast$. Such an element belongs to $G_-$, if and only if the corresponding function $\mu_\omega=f(\lambda)+g(\lambda)\nu$ obeys~\eqref{involutions mu}. It belongs to $G_+$, if and only if the corresponding function $\mu=f(\lambda)+g(\lambda)\nu$ obeys $\sigma^\ast\mu=\mu^{-1}$, extends holomorphically to $\lambda=0$ and equals one at that point.
\end{proof}
The proof shows that the function $\mu_\omega$ in the lemma is unique and equal to the action of $M_\omega$ on the eigenspaces of $\zeta_0$. The logarithmic derivative of this function is a meromorphic differential of second kind with second order poles at $\lambda=0$ and $\lambda=\infty$. Due to~\eqref{involutions mu} it takes the form
\begin{align}\label{def:b}
d\ln\mu_\omega&=\frac{b_\omega(\lambda)}{2\nu}d\ln\lambda&\text{with}&&b_\omega&\in\mathbb{C}^3[\lambda]&\text{such that}&&\bar{\lambda}^3b_\omega(\bar{\lambda}^{-1})&=-\overline{b_\omega(\lambda)}.
\end{align}
The forgoing lemma implies that a branch of $\ln\mu_\omega$ obeys $\ln\mu_\omega=\omega\lambda^{-1}\nu+\mathbf{O}(\nu)$ nearby $\lambda=0$. Since $a(0)=1$ we have $\nu^2=-\lambda+\mathbf{O}(\lambda^2)$ and $d\nu=\frac{-1+\mathbf{O}(\lambda)}{2\nu}d\lambda$ nearby $\lambda=0$ and $b_\omega(0)=\omega$ as well as the fact that the highest coefficient of $b_\omega$ equals $-\bar{\omega}$. Since the integrals along the closed cycles uniquely determine the 1-forms on $\Sigma^\ast$ (which extend holomorphically to $\lambda=0$ and $\lambda=\infty$) the two other coefficients of $b_\omega$ are determined by the condition that the integrals of $d\ln\mu_\omega$ along the closed cycles of $\Sigma^\ast$ are purely imaginary. If all these integrals are multiples of $2\pi i$, then there exists a function $\mu_\omega$, whose logarithmic derivative is given by~\eqref{def:b}. The 1-form $d\ln\mu_\omega$ determines $\mu_\omega$ only up to a multiplicative constant and~\eqref{involutions mu} up to $\pm 1$. But the characterization of $\mu_\omega$ in Lemma~\ref{characterisation 1} determines $\mu_\omega$ uniquely in terms of $b_\omega$. This proves the following corollary:
\begin{corollary}\label{characterisation 2}
For $a\in\mathcal{M}_2^1$ the elements of $\Gamma_a$ are the values $\omega=b_\omega(0)$ of those $b_\omega$ in~\eqref{def:b}, whose 1-forms $d\ln\mu_\omega$ are the logarithmic derivatives of a holomorphic function $\mu_\omega$ on $\Sigma^\ast$~\eqref{eq:spectral curve}.\qed
\end{corollary}
The arguments of the proof of Theorem~\ref{thm orbits} can be used to extend the characterization in Lemma~\ref{characterisation 1} to a characterization of $\Gamma_\zeta$ for all $\zeta\in\mathcal{P}_2$. Due to Theorem~\ref{thm orbits}, the groups $\Gamma_\zeta$ are no lattices for $a\in\mathcal{M}_2^2\cup\mathcal{M}_2^3\cup\mathcal{M}_2^4\cup\mathcal{M}_2^5$. For $a\in\mathcal{M}_2^2\cup\mathcal{M}_2^3$ we impose in addition to condition~\eqref{involutions mu}
\begin{align}\label{eq:regular function}
\mu_\omega&=f_\omega(\lambda)+g_\omega(\lambda)\nu&\text{with holomorphic}&&
f_\omega,g_\omega&:\mathbb{C}\setminus\{0\}\to\mathbb{C}.
\end{align}
This condition ensures that $\mu_\omega$ is regular on the variety $\Sigma^\ast$~\eqref{eq:spectral curve} which has singularities at the higher order roots of $a$. Together with~\eqref{involutions mu} it implies that $\mu_\omega$ takes the value $\pm 1$ at each root of $a$.\vspace{3mm}

For each $a\in\mathcal{M}_2^1\cup\mathcal{M}_2^2\cup\mathcal{M}_2^3$ let $\Sigma^\circ$ denote the variety~\eqref{eq:spectral curve} without the singular points at the double roots of $a$. For such $a$ we choose two closed cycles of $\Sigma^\circ$ which surround in the $\lambda$-plane exactly two roots $\alpha$ and $\bar{\alpha}^{-1}$ of $\lambda\mapsto\lambda a(\lambda)$ and call them $A$-cycles. The sum of both negative $A$-cycles is homologous to a cycle surrounding the two fixed points of $\sigma$ at $\lambda=0$ and $\lambda=\infty$. If two simple roots of $a$ coalesce at a unitary double root, then the corresponding $A$-cycle converges to the cycle of $\Sigma^\circ$ which surrounds in the $\lambda$-plane the double root. The involution $\rho$ maps these $A$-cycles onto their negative, since they contain two fixed points of $\rho$. 
\begin{lemma}\label{1-forms}
For all $a\in\mathcal{M}_2^1\cup\mathcal{M}_2^2\cup\mathcal{M}_2^3$ and $\omega\in\mathbb{C}$ there exists a unique $b_\omega\in\mathbb{C}^3[\lambda]$ with the following properties:
\begin{enumerate}
\item $b_\omega(0)=\omega$.\label{cond 1}
\item $\bar{\lambda}^3b_\omega(\bar{\lambda}^{-1})=-\overline{b_\omega(\lambda)}$ for all $\lambda\in\mathbb{C}$.\label{cond 2}
\item The integral of $\frac{b_\omega(\lambda)}{2\nu}d\ln\lambda$ along the $A$-cycles vanishes.\label{cond 3}
\end{enumerate}
\end{lemma}
\begin{proof}
The integrals in the third condition are real, because the second condition implies the equations
\begin{align}\label{eq:involution}
\sigma^\ast\tfrac{b_\omega(\lambda)}{2\nu}d\ln\lambda&=-\tfrac{b_\omega(\lambda)}{2\nu}d\ln\lambda,&
\rho^\ast\tfrac{b_\omega(\lambda)}{2\nu}d\ln\lambda&=-\tfrac{\overline{b_\omega(\lambda)}}{2\nu}d\ln\bar{\lambda}.
\end{align}
Condition~\ref{cond 1} and \ref{cond 2} imply that $b_\omega$ is equal to
\begin{align*}
b_\omega&=\omega-\bar{\omega}\lambda^3+\beta_1(\lambda-\lambda^2)+i\beta_2(\lambda+\lambda^2)&\text{with}&&\beta_1,\beta_2&\in\mathbb{R}.
\end{align*}
We convert the third condition into a real linear inhomogeneous system of two equations with respect to $\beta_1$ and $\beta_2$. Let $\mathbf{A}$ denote the corresponding real two-by-two coefficient matrix. Now the assertion follows if $\mathbf{A}$ is invertible. For $a\in\mathcal{M}_2^1$ the following 1-forms corresponding to $\beta_1$ and $\beta_2$ build a basis of the holomorphic 1-forms of the compact Riemann surface $\bar{\Sigma}$ of genus two which is the two-sheeted covering of $\mathbb{P}^1$ branched at the four simple roots of $a$ as well as at $\lambda=0$ and $\lambda=\infty$:
\begin{align*}
\frac{\lambda-\lambda^2}{\nu}d\ln\lambda&&\frac{i(\lambda+\lambda^2)}{\nu}d\ln\lambda.
\end{align*}
This implies that $\mathbf{A}$ is invertible. For $a\in\mathcal{M}_2^2$ let $\bar{\Sigma}$ denote the compact Riemann surface of genus one which is the two-sheeted covering branched at the 2 simple roots of $a$ and at $\lambda=0$ and $\lambda=\infty$. The 1-forms of those $b_\omega\in\mathbb{C}^2[\lambda]$, which vanish at the double root of $a$ and at $\lambda=0$ extend to holomorphic 1-forms on $\bar{\Sigma}$. The integral of these 1-forms along the $A$-cycle surrounding the two simple roots of $a$ is non-zero. The integral along the other $A$-cycle around the double root of $a$ is a multiple of the value of $b_\omega$ at the double root. This again implies that $\mathbf{A}$ is invertible. Finally, for $a\in\mathcal{M}_2^3$ the integrals along an $A$-cycle surrounding a double root is a multiple of the value of $b_\omega$ at this double root. Therefore, also in this case $\mathbf{A}$ is invertible. This finishes the proof.
\end{proof}
Our choice of $A$-cycles is defined in a neighborhood of each $a\in\mathcal{M}_2^1\cup\mathcal{M}_2^2\cup\mathcal{M}_2^3$. The entries of $\mathbf{A}$ and the integrals of the 1-forms corresponding to $b_\omega=\omega-\bar{\omega}\lambda^3$ along the $A$-cycles depend continuously on $(\omega,a)\in\mathbb{C}\times\mathcal{M}_2^1\cup\mathcal{M}_2^2\cup\mathcal{M}_2^3$. Therefore, the same is true for $b_\omega$.\vspace{3mm}

Now we supplement the two $A$-cycles by two $B$-cycles. For $a\in\mathcal{M}_2^1$ let $\alpha_i$ with $i=1, 2$ denote the unique simple roots of $a$ with $|\alpha_i|<1$, which is surrounded by the $i$-th $A$-cycle. The first $B$-cycle is defined as a cycle that surrounds in the $\lambda$-plane inside the unit disc $\lambda=0$ and passes through $\alpha_1$ without surrounding the other simple roots inside the unit disc. This cycle is homologous to a cycle that surrounds $\lambda=0$ and $\lambda=\alpha_1$ inside the unit disc. Our definition is also well defined in the limiting cases $a \to \mathcal{M}_2^2 \cup \mathcal{M}_2^3 \cup \mathcal{M}_2^4 \cup \mathcal{M}_2^5$. The other $B$-cycle passes through $\bar{\alpha}_2^{-1}$ outside the unit disc and surrounds all other roots of $a$. Again, this cycle is homologous to a cycle that surrounds $\alpha_1, \alpha_2, \bar{\alpha}_1^{-1}$. Together with the $A$-cycles they build a canonical basis of cycles of the compact Riemann surface $\bar{\Sigma}$ of genus two. The involution $\rho$ reverses the orientation and the intersection numbers. Therefore $\rho$ preserves the $B$-cycles up to an addition of $A$-cycles and the integrals of the 1-forms~\eqref{def:b} in Lemma~\ref{1-forms} along the $B$-cycles are purely imaginary. Due to the first equation in~\eqref{eq:involution} these 1-forms have no residues at $\lambda=0$ and $\lambda=\infty$ and all their integrals are purely imaginary. For $a \in \mathcal{M}_2^2\cup\mathcal{M}_2^3$ let $\bar{\Sigma}$ denote the smooth compact Riemann surface of genus $0$ or $1$ which is the two-sheeted covering branched only at the simple roots of $a$ and at $\lambda=0$ and $\lambda=\infty$. In the limit $a\to\mathcal{M}_2^2\cup\mathcal{M}_2^3$ a $B$-cycle passing through a root of $a$ which coalesce at a unitary double root converges to a non-closed path of $\bar{\Sigma}$. This path surrounds $\lambda=0$ and connects the two points in $\bar{\Sigma}$ lying above the double root of $a$. The conclusion that the integrals of the 1-forms~\eqref{def:b} in Lemma~\ref{1-forms} along the $B$-cycles and all other cycles of $\Sigma^\circ$ are purely imaginary is also true for $a\in\mathcal{M}_2^2\cup\mathcal{M}_2^3$. This implies that for all $\omega\in\mathbb{C}$ there exists a harmonic function $h_\omega:\Sigma^\circ\to\mathbb{R}$, such that $dh$ is the real part of the 1-form~\eqref{def:b}. The derivative determines $h_\omega$ up to an additive constant, which is fixed by the condition $\sigma^\ast h_\omega=-h_\omega$.
\begin{lemma}\label{continuous extension}
The map $(\omega,a)\mapsto b_\omega$ has a continuous extension to $(\omega,a)\in\mathbb{C}\times\mathcal{M}_2$. 
\end{lemma}
\begin{proof}
For any sequence $(a_n)_{n\in\mathbb{N}}$ in $\mathcal{M}_2^1\cup\mathcal{M}_2^2\cup\mathcal{M}_2^3$ with limit in $\mathcal{M}_2^4\cup\mathcal{M}_2^5$ the corresponding renormalized polynomials $b_\omega/\|b_\omega\|$ have convergent subsequences. Here, $\|\cdot\|$ is any norm defined on the vector space $\mathbb{C}^3[\lambda]$. The corresponding functions $h_\omega/\|b_\omega\|$ also converge uniformly on compact subsets of $\Sigma^\circ$. By the strong maximum principle for harmonic functions, $h_\omega$ is bounded on the complement of two small disjoint discs around the two poles of $dh_\omega$ at $\lambda=0$ and $\lambda=\infty$, since the singular parts of $dh$ are determined by $\omega$. Due to the theorem on removable singularities of harmonic functions, the limit extends to a harmonic function on the two-sheeted covering of $\mathbb{C}\setminus\{0\}$, which is branched only at the simple roots of $a$. This implies that the limit of $b_\omega/\|b_\omega\|$ vanishes at all double roots of $a$ and has a second order root at any fourth order root of $a$. In particular, the coefficients $\beta_1$ and $\beta_2$ of $b_\omega$ are bounded for bounded $\omega$. Therefore, $\|b_\omega\|$ is bounded and all limits of $b_\omega$ coincide. This implies the claim.
\end{proof}
Now we are ready to prove the main theorem of this section.
\begin{theorem}\label{continuous T}
For $a\in\mathcal{M}_2^1\cup\mathcal{M}_2^2\cup\mathcal{M}_2^3$ the values $b_\omega(0)=\omega$ of all $b_\omega$ in~\eqref{def:b} with the following property build a lattice $\tilde{\Gamma}_a$ in $\mathbb{C}$: They define $d\ln\mu_\omega$~\eqref{def:b} of a function $\mu_\omega$ on $\Sigma^\ast$~\eqref{eq:spectral curve} which obeys~\eqref{involutions mu} and~\eqref{eq:regular function}. The map $T:\mathcal{M}_2^1\cup\mathcal{M}_2^2\cup\mathcal{M}_2^3\to\mathcal{F}$ to the corresponding isomorphism class is continuous. For all compact subsets $\mathcal{K}\subset\mathcal{F}$ each $a\in\mathcal{M}_2^4\cup\mathcal{M}_2^5$ has a neighborhood $O$ in $\mathcal{M}_2$, such that $T^{-1}[\mathcal{K}]\cap O=\emptyset$.
\end{theorem}
The last statement ensures that the map $T$ extends to a continuous map from $\mathcal{M}_2$ to the one-point compactification $\mathcal{F}\cup\{\infty\}$ of $\mathcal{F}$, such that $T$ takes the value $\infty$ on $\mathcal{M}_2^4\cup\mathcal{M}_2^5$.
\begin{proof}
For $a\in\mathcal{M}_2^1\cup\mathcal{M}_2^2\cup\mathcal{M}_2^3$ the condition~\eqref{eq:regular function} is equivalent to the condition that $\mu_\omega$ takes only one value at the double roots of $a$. For $\mu_\omega$ obeying~\eqref{involutions mu}, this condition is equivalent to the condition that $\mu_\omega$ takes values $\pm1$ at the double roots of $a$. Therefore, an $\omega\in\mathbb{C}$ has the property in the Theorem~\ref{continuous T} if and only if all integrals of the 1-forms~\eqref{def:b} in Lemma~\ref{1-forms} along the $B$-cycles belong to $2\pi i\mathbb{Z}$. Now we claim that if these integrals along the $B$-cycles vanish, then $\omega$ vanishes. In fact, if $\omega\ne0$ and these integrals vanish, then the integrals along all cycles vanish and the 1-form~\eqref{def:b} is the exterior derivative of a meromorphic function $h$. For $a\in\mathcal{M}_2^1$ the function $h-\sigma^\ast h$ has exactly two simple poles at $\lambda=0,\infty$ and roots at all four roots of $a$, which is impossible. For $a\in\mathcal{M}_2^2\cup\mathcal{M}_2^3$ the polynomial $b_\omega$ vanishes at the double roots of $a$ and $h$ is holomorphic on the two-sheeted covering of $\mathbb{C}\setminus\{0\}$, which is branched only at the simple roots of $a$. Again $h-\sigma^\ast h$ has simple poles at $\lambda=0,\infty$ and vanishes at simple roots of $a$ and at both points which cover a double root of $a$, which is impossible. This proves the claim.\vspace{3mm}

In particular, for $a\in\mathcal{M}_2^1\cup\mathcal{M}_2^2\cup\mathcal{M}_2^3$ the map from $\omega\in\mathbb{C}$ to the integrals of the 1-forms~\eqref{def:b} in Lemma~\ref{1-forms} along the $B$-cycles is a $\mathbb{R}$-linear isomorphism to $i\mathbb{R}^2$. We call these integrals $B$-periods. This shows that the property of Theorem~\ref{continuous T} characterizes the elements $\omega$ of a lattice. Again, the $B$-cycles extend to small neighborhoods in $\mathcal{M}_2^1\cup\mathcal{M}_2^2\cup\mathcal{M}_2^3$ and the $B$-periods depend continuously on $(\omega,a)\in\mathbb{C}\times\mathcal{M}_2^1\cup\mathcal{M}_2^2\cup\mathcal{M}_2^3$. Therefore $T:\mathcal{M}_2^1\cup\mathcal{M}_2^2\cup\mathcal{M}_2^3\to\mathcal{F}$ is well defined and continuous.\vspace{3mm}

It remains to prove the last statement. We first consider limits $a\to\mathcal{M}_2^4$. A small neighborhood in $\mathcal{M}_2$ of each $a\in\mathcal{M}_2^4$ is disjoint from $\mathcal{M}_2^2\cup\mathcal{M}_2^3$. For $a\in\mathcal{M}_2^1$ the lattices $\tilde{\Gamma}_a=\Gamma_a$ contain those $\omega\in\mathbb{C}$, whose $B$-periods belong to $2\pi i\mathbb{Z}$. We claim that the $\omega\in\tilde{\Gamma}_a$ with equal $B$-periods are bounded in the limit $a\to\mathcal{M}_2^4$, and the $\omega\in\tilde{\Gamma}_a$ with different $B$-periods are unbounded in the limit $a\to\mathcal{M}_2^4$.\vspace{3mm}

Let us first assume on the contrary that there exists a convergent sequence $a_n\to a\in\mathcal{M}_2^4$ in $\mathcal{M}_2^1$ and a bounded sequence $\omega_n\in\tilde{\Gamma}_{a_n}$ with constant different $B$-periods. Due to Lemma~\ref{continuous extension} the sequence $b_{\omega_n}$ of a convergent subsequence $\omega_n$ converges. Moreover the corresponding meromorphic 1-forms~\eqref{def:b} in Lemma~\ref{1-forms} converge to a meromorphic 1-form on the two-sheeted covering of $\mathbb{P}^1$, which is branched only at $\lambda=0$ and $\lambda=\infty$. The limits have two simple poles at $\lambda=0$ and $\lambda=\infty$ and no other poles. Along cycles surrounding two simple roots of $a_n$, which coalesce at $a\in\mathcal{M}_2^4$, the integrals of the 1-forms~\eqref{def:b} defined by $b_{\omega_n}$ converge to zero. For $a_n\in\mathcal{M}_2^1$ these cycles have opposite intersection numbers $\pm 1$ with both $A$-cycles, and these integrals are the difference of both $B$-periods. This contradicts the assumption of constant different $B$-periods. Therefore, the $\omega\in\Gamma_a$ with different $B$-periods are unbounded.\vspace{3mm}

Now we show that all $\omega\in\Gamma_a$ with equal $B$-periods extend continuously to $a\to\mathcal{M}_2^4$. If the $B$-periods coincide, then the corresponding 1-form~\eqref{def:b} in Lemma~\ref{1-forms} has vanishing integrals along all cycles of the two-sheeted covering over a disc in $\lambda\in\mathbb{C}\setminus\{0\}$ containing all four roots of $a$. In particular, this 1-form is the derivative $dp_\omega$ of a holomorphic function $p_\omega$ on this two-sheeted covering. Let $a_n\to\mathcal{M}_2^4$ be a convergent sequence in $\mathcal{M}_2^1$ and $\omega_n\in\tilde{\Gamma}_{a_n}$ a sequence with constant equal $B$-periods. The sequence of functions $p_{\omega_n}/|\omega_n|$ converges locally on the two-sheeted covering nearby the four roots of $a$. Consequently $b_{\omega_n}/|\omega_n|$ and the integrals of the corresponding 1-forms~\eqref{def:b} along the $B$-periods converge to the corresponding integrals of the limit of $b_{\omega_n}/|\omega_n|$. An explicit calculation yields that the limit of $b_{\omega_n}$ in terms of the constant equals $B$-periods. Hence, $\omega_n$ are bounded and the $\omega\in\tilde{\Gamma}_a$ extends continuously to the limits $a\to\mathcal{M}_2^4$. This proves the claim and for $a\to\mathcal{M}_2^4$ the last statement of the theorem.\vspace{3mm}

These arguments extends to limits $a\to\mathcal{M}_2^5$. For bounded sequences $b_{\omega_n}\in\tilde{\Gamma}_{a_n}$ with constant different $B$-periods and $a_n\to a\in\mathcal{M}_2^5$ a subsequence of the functions~\eqref{eq:regular function} $f_{\omega_n}$ and $g_{\omega_n}$ converge by the maximum modulus theorem on all compact subsets of $\lambda\in\mathbb{C}\setminus\{0\}$. In particular, the corresponding $B$-periods converge to the integrals of the limiting 1-forms along the limits of the $B$-cycles. By definition of the $B$-cycles and due to the transformation properties of $b_{\omega}$ the limits of the periods are again the same in contradiction to constant different $B$-periods.\vspace{3mm}

For sequences $\omega_n\in\tilde{\Gamma}_{a_n}$ with constant equal $B$-periods, the functions $p_{\omega_n}$ take the same value at all four roots of $a_n$. Without loss of generalization we may choose this value as zero. Then $p_{\omega_n}$ can be represented as $h_{\omega_n}(\lambda)\nu$ with a holomorphic function $h_{\omega_n}$ on the disc in $\lambda\in\mathbb{C}\setminus\{0\}$ around the four roots of $a_n$. Again the function $h_{\omega_n}/|\omega_n|$ converges. In particular, the limit of $p_{\omega_n}/|\omega_n|$ has a second order root at the fourth order root of the limit $a\in\mathcal{M}_2^5$. Consequently, the limit of $b_{\omega_n}/|\omega_n|$ has a third order root there and the limit of $b_{\omega_n}$ is again determined by the constant equal $B$-periods. Therefore, also in this case the $\omega\in\tilde{\Gamma}_a$ with equal $B$-periods extend continuously to the limits $a\to\mathcal{M}_2^5$.
\end{proof}

\section{Explicit calculation of spectral data in $\mathcal{M}_2^2$ and $\mathcal{M}_2^3$.} \label{Sec:g1}

In this section we shall express for $a\in\mathcal{M}_2^2\cup\mathcal{M}_2^3$ the period lattices $\tilde{\Gamma}_a$ and the corresponding 1-forms~\eqref{def:b} in Lemma~\ref{1-forms} in terms of elliptic functions. We start with $a\in\mathcal{M}_2^2$ and afterwards consider $a\in\mathcal{M}_2^3$ as limiting cases.\vspace{3mm}

We parameterize the elements $a\in\mathcal{M}_2^2$ by $r\in(0,1)$ and $\varphi\in[0,\pi)$ such that
$$a(\lambda)=(\lambda-re^{-2i\varphi})(\lambda-r^{-1}e^{-2i\varphi})(\lambda-e^{2i\varphi})^2.
$$
The two-sheeted covering $\bar{\Sigma}$ is branched at the two simple roots of $a$ and an elliptic curve at $\lambda=0$ and $\lambda=\infty$. We use elliptic Weierstra{\ss} functions~\cite[Chapter~13.12]{Bat} and present this elliptic curve as $z\in\mathbb{C}/(2\omega\mathbb{Z}+2\omega'\mathbb{Z})$. Since $\bar{\Sigma}$ is endowed with two involutions~\eqref{1-forms} whose composition has no fixed points, we may choose the first half period $\omega$ to be real and the second half period to be imaginary. The involutions act as
\begin{align}\label{eq:involutions z}
\sigma&:z\mapsto -z,&\rho&:z\mapsto\bar{z}+\omega'.
\end{align}
We identify $z=0$ with $\lambda=\infty$ and $z=\omega'$ with $\lambda=0$. For simplicity, we determine $\hat{\lambda}$ and the eigenvalue $\hat{\nu}$ of $\hat{\zeta}$ instead of $(\lambda,\nu)$ in terms of Weierstra{\ss} elliptic functions. They have the following properties:
\begin{align}
\hat{\nu}^2+\hat{\lambda}\hat{a}(\hat{\lambda})&=0&\mbox{with}\qquad
\hat{a}(\hat{\lambda})&=(\hat{\lambda}+r)(\hat{\lambda}+r^{-1})\in\mathcal{M}_1,\label{hat elliptic curve}\\
\sigma:(\hat{\lambda},\hat{\nu})&\mapsto(\hat{\lambda},-\hat{\nu}),&
\rho:(\hat{\lambda},\hat{\nu})&\mapsto(\bar{\hat{\lambda}}^{-1},-\bar{\hat{\lambda}}^{-2}\bar{\hat{\nu}}).\nonumber
\end{align}
These properties determine the values of $\wp$ at the roots of $\wp'$ and $\hat{\lambda}$ and $\hat{\nu}$:
\begin{align}\label{r and e1 e2 e3}
e_3=-\tfrac{r+r^{-1}}{3}<e_2=\tfrac{2r-r^{-1}}{3}&<e_1=\tfrac{2r^{-1}-r}{3},&
\hat{\lambda}(z)&
=e_3-\wp(z),&
\hat{\nu}&
=\frac{\wp'(z)}{2}.
\end{align}
The inequalities are given in~\cite[Chapter~13.15]{Bat}. The Weierstra{\ss} invariants are equal to
\begin{align*}
g_2(r)&:=\tfrac{4}{3}(r+r^{-1})^2-4,&
g_3(r)&:=\tfrac{8}{27}(r+r^{-1})^3-\tfrac{4}{3}(r+r^{-1}).
\end{align*}
The corresponding $a$ has an additional double root in the fixed point set $\hat{\lambda}\in\mathbb{S}^1$ of $\rho$. This fixed point set decomposes into the two components $z\in\pm\frac{\omega'}{2}+\mathbb{R}/2\omega\mathbb{Z}$ denoted by III and IV in Figure~\ref{fig:Imtau_eq_omegahalbe}.
\begin{figure}[h]
	\centering
	\captionsetup{width=0.68\textwidth}
	\fbox{
		\begin{minipage}{11 cm}
			\centering
			\begin{tikzpicture}[scale=1.0]			
			\draw[thick] (0,-1) rectangle (4,1);
			\draw[thick,color=red] (0,0) -- (4,0)node[right]{$I$};
			\draw[thick,color=blue] (0,1) -- (4,1)node[right]{$II$};
			\draw[thick,color=blue] (0,-1) -- (4,-1)node[right]{$II$};
			\draw[thick,color=green] (0,0.5) -- (4,0.5)node[right]{$III$};
			\draw[thick,color=green] (0,-0.5) -- (4,-0.5)node[right]{$IV$};
			\draw[thick] (2,-1)node[below]{$\omega-\omega'$}--(2,1)node[above]{$\omega+\omega'$};
			\path (0,1)node[left]{$\omega'$};	
			\path (0,0)node[left]{$0$};	
			\path (0,-1)node[left]{$-\omega'$};
			\path (2,0)node[above left]{$\omega$};
			\filldraw (0,0) circle (1.5pt);
			\filldraw (0,1) circle (1.5pt);
			\filldraw (0,-1) circle (1.5pt);
			\filldraw (2,0) circle (1.5pt);
			\filldraw (2,1) circle (1.5pt);
			\filldraw (2,-1) circle (1.5pt);
			\end{tikzpicture}
			\caption{Fixed point sets $I,II$ of the involution $z\mapsto\bar{z}$ and $III,IV$ of the involution $\rho$ on the torus $\mathbb{C}/(2\omega\mathbb{Z}+2\omega'\mathbb{Z})$.}
			\label{fig:Imtau_eq_omegahalbe}
		\end{minipage}
	}
\end{figure}
Instead of $\varphi$ we shall use $t\in\mathbb{R}/2\omega\mathbb{Z}$ such that $z_+=\frac{\omega'}{2}+t$ corresponds to the double root of $a$ with $\hat{\lambda}(z_+)=-e^{4i\varphi}$. The parameter $\lambda$ and $a\in\mathcal{M}_2$ are equal to
\begin{align*}
\lambda&=-e^{-2i\varphi}\hat{\lambda},&
a(\lambda)&=(\lambda+e^{-2i\varphi}\hat{\lambda}_1)(\lambda+e^{-2i\varphi}\hat{\lambda}_1^{-1})(\lambda-e^{2i\varphi})^2.
\end{align*}
In particular, we have
\begin{align*}
d\ln\lambda&=\frac{\wp'(z)dz}{\wp(z)-e_3},&dz=\frac{-\hat{\lambda}}{2\hat{\nu}}d\ln\hat{\lambda}.
\end{align*}
The subgroup $\Gamma_a\subset\mathbb{C}$~\eqref{lattice a} contains a one-dimensional subspace. This subspace contains a one-dimensional lattice of periods in $\tilde{\Gamma}_a$. Let $\mu_1$ denote the eigenvalue of a monodromy along a generator of this lattice. The logarithm is a single-valued meromorphic function on $\mathbb{C}/(2\omega\mathbb{Z}+2\omega'\mathbb{Z})$ (compare~\cite[Chapter~13.13~(18)]{Bat}):
\begin{equation}\label{lnmu1}
\ln\mu_1=\pi i\frac{\zeta(z)-\zeta(z-\omega')-\eta'}{\zeta(z_+)-\zeta(z_+-\omega')-\eta'}=\pi i\frac{\wp'(z)}{2(e_3-\wp(z))}\frac{2(e_3-\wp(z_+))}{\wp'(z_+)}=\pi i\frac{\hat{\nu}}{\hat{\lambda}}\frac{\hat{\lambda}_+}{\hat{\nu}_+},
\end{equation}
where index $+$ denotes the value at $z_+$. This meromorphic function has first order poles at $z\in 2\omega\mathbb{Z}+\omega'\mathbb{Z}$ with $\ln\mu_{1,+}=\pi i$. The involutions act on $\ln\mu_1$ as (compare \eqref{involutions mu}):
\begin{align*}
\sigma^\ast\ln\mu_1&=-\ln\mu_1,&\rho^\ast\ln\mu_1&=-\overline{\ln\mu}_1.
\end{align*}
These properties uniquely determine $\ln\mu_1$ with roots at $z\in\omega+2\omega\mathbb{Z}+\omega'\mathbb{Z}$.

Let $\mu_2$ denote the eigenvalue of another monodromy, whose logarithmic derivative $d\ln\mu_2$ has integral $2\pi i$ along $2\omega$. The meromorphic function \begin{equation}\label{lnmu2}
\ln\mu_2=\omega'(\zeta(z)\!+\!\zeta(z\!-\!\omega')\!+\!\eta')\!-\!2\eta'z\!-\!\frac{\omega'(\zeta(z_+)\!+\!\zeta(z_+\!-\!\omega')\!+\!\eta')\!-\!2\eta'z_+}{\pi i}\ln\mu_1
\end{equation}
is multivalued on $\bar{\Sigma}$ but single valued on $z\in\mathbb{C}$ with first order poles at $z\in 2\omega\mathbb{Z}+\omega'\mathbb{Z}$. It takes the value $\pi i$ at $z=\omega$ and is uniquely determined by the following properties:
\begin{align*}
\ln\mu_2(z+2\omega)&=\ln\mu_2(z)+2\pi i,&
\ln\mu_2(z+2\omega')&=\ln\mu_2(z),&
\ln\mu_2(z_+)&=0,\\
\ln\mu_2(-z)&=-\ln\mu_2(z),&
\ln\mu_2(\bar{z}+\omega')&=-\overline{\ln\mu_2(z)}.
\end{align*}
The corresponding periods $\omega_1,\omega_2$ generate $\tilde{\Gamma}_a$ and are proportional to the Laurent coefficients of $\ln\mu_1$ and $\ln\mu_2$ in front of $\frac{1}{z-\omega'}$ at $z=\omega'$. Therefore $\tilde{\Gamma}_a$ is isomorphic to
$$\frac{-\pi i}{\zeta(z_+)-\zeta(z_+-\omega')-\eta'}\mathbb{Z}+\left(\omega'+\frac{\omega'(\zeta(z_+)+\zeta(z_+-\omega')+\eta')-2\eta'z_+}{\zeta(z_+)-\zeta(z_+-\omega')-\eta'}\right)\mathbb{Z}.$$
This lattice is isomorphic to the lattice generated by $1$ and 
$$
\tilde{\tau}_a=\frac{2\eta'z_+-2\omega'\zeta(z_+)}{\pi i}.$$
Now we consider the limit $r\to1$ which corresponds to $a\to\mathcal{M}_2^3$. In this case we have
\begin{align*}
e_3&=-\tfrac{2}{3},&e_2&=\tfrac{1}{3},&e_1&=\tfrac{1}{3},&g_2&=\tfrac{4}{3},&g_3&=-\tfrac{8}{27}.
\end{align*}
Due to~\cite[Chapter~13.15]{Bat} this corresponds to
\begin{gather}\nonumber\begin{aligned}
\omega&=\infty,&\omega'&=\frac{\pi i}{2},&\wp(z)&=\frac{1}{3}+\frac{1}{\sinh^2(z)},&\zeta(z)&=-\frac{z}{3}+\coth(z),\\
\eta'&=-\frac{\pi i}{6},&z_+&=\frac{\pi i}{4}+t,&\hat{\lambda}(z)&=-\frac{1}{\sinh^2(z)},&\hat{\nu}&=i\frac{\cosh(z)}{\sinh^3(z)},
\end{aligned}\\\label{genus0}\begin{aligned}
\ln\mu_1&=\pi i\frac{\cosh(2t)}{\sinh(2z)},&\ln\mu_2&=\pi i\frac{\cosh(2z)-\sinh(2t)}{\sinh(2z)},&\tilde{\tau}_a&=\frac{i-\tanh(t)}{1-i\tanh(t)}.
\end{aligned}\end{gather}
	\begin{figure}[h]
		\centering
		\captionsetup{width=0.53\textwidth}
		\fbox{
		\begin{minipage}{8.6 cm}
				\centering
		\begin{tikzpicture}[scale=1.2]
		\clip (-1.6,-0.1) rectangle (1.9,2.9);
		\fill[fill=gray!30]
		(-0.5,2.4) -- (120:1cm) arc (120:60:1cm) --(0.5,2.4)-- cycle;
		\draw[->,thin] (-1.5,0) -- (1.5,0) node[right] {$\Re$};
		\draw[->,thin] (0,-0.1) -- (0,2.5) node[above] {$\Im$};
		\draw[dashed,thick] (0,0) circle (1cm);
		\draw[dashed,thick] (0.5,-0.1)--(0.5,2.4);
		\draw[dashed,thick] (-0.5,-0.1)--(-0.5,2.4);
		\draw[very thick,color=orange](1,0) arc (0:60:1cm);
		\draw[->,very thick,color=orange](60:1cm) arc (60:30:1cm);
		\draw[very thick,color=green](60:1cm) arc (60:120:1cm);
		\draw[very thick,color=blue](120:1cm) arc (120:180:1cm);
		\draw[very thick,color=blue](60:1cm) -- (0.5,2.4);
		\draw[->,very thick, color=blue](60:1cm) -- (0.5,1.6);
		\draw[very thick,color=orange](120:1cm) -- (-0.5,2.4);
		\draw[->,very thick, color=orange](120:1cm) -- (-0.5,1.6);
		\draw[->,very thick, color=blue](120:1cm) arc (120:150:1cm);
		\path (0,1.7) node[fill=gray!30]{$\mathcal{F}$};
		\end{tikzpicture}
		\caption{$\tilde{\tau}_a$ for $a\in\mathcal{M}_2^3$ and the corresponding elements of $\mathcal{F}$ are sketched in the same color.}
		\label{fig:doubleroot_tau_mapping}
	\end{minipage}
	}
	\end{figure}
\begin{theorem}\label{thm immersion}
$T$ restricted to $\mathcal{M}_2^2$ is a surjective immersion onto $\mathcal{F}\setminus\{\frac{1\pm i\sqrt{3}}{2}\}$.
\end{theorem}
\begin{remark}
The restriction of $T$ to $\mathcal{M}_2^2\cap T^{-1}[\{\tau \in \mathcal{F} \mid |\tau|>1 \text{ and } \Re(\tau) \in (-\tfrac{1}{2}, \tfrac{1}{2})\}]$ is a 3-sheeted unbranched covering. Moreover, $T$ restricted to $\mathcal{M}_2^3$ is a surjective map to $\{\tau \in \mathcal{F} \mid |\tau|=1 \text{ or } \Re(\tau)=\pm\tfrac{1}{2}\}$ that is an immersion on the open, dense subset $\mathcal{M}_2^3 \cap T^{-1}[\mathcal{F}\setminus\{\frac{1\pm i\sqrt{3}}{2}, i\}]$. In Figure~\ref{fig:doubleroot_tau_mapping} we show $\tilde{\tau}_a$~\eqref{genus0} for $a\in\mathcal{M}_2^3$ and the corresponding elements in the boundary of $\mathcal{F}$.
\end{remark}
\begin{proof}
First we show that the restriction of $T$ to $\mathcal{M}_2^2$ is an immersion. For $z_+\in\frac{\omega'}{2}+\mathbb{R}$ we use $\bar{z}_+=z_+-\omega'$ and obtain
\begin{align*}
\Re(\tilde{\tau}_a)&\!=\!\frac{2\eta'z_+\!-\!\omega'(\zeta(z_+)\!+\!\zeta(z_+\!-\!\omega')\!+\!\eta')}{\pi i},&\Im(\tilde{\tau}_a)&\!=\!\frac{\omega'(\eta'\!-\!\zeta(z_+)\!+\!\zeta(z_+\!-\!\omega')}{-\pi}\!=\!\frac{\omega'\hat{\nu}_+}{\pi\hat{\lambda}_+}.
\end{align*}
The second expression is a meromorphic function on the elliptic curve~\eqref{hat elliptic curve}, and the first obeys
\begin{align*}
\Re(\tilde{\tau}_a)(z_++2\omega)&=\Re(\tilde{\tau}_a)(z_+)-2,&
\Re(\tilde{\tau}_a)(z_++2\omega')&=\Re(\tilde{\tau}_a)(z_+).
\end{align*}
We consider $\big(\Re(\tilde{\tau}_a),\Im(\tilde{\tau}_a)\big)$ and their complex continuations as functions depending on $(\hat{\lambda}_+,r)\in\mathbb{S}^1\times(0,1)\subset\mathbb{C}\times\mathbb{R}$. Consequently, their derivatives have poles at $\hat{\lambda}_+\in\{\infty,0\}$ and at the roots of $d\hat{\lambda}_+$ at $\hat{\lambda}_+\in\{-r,-r^{-1}\}$. Since the additive constant in the first equation does not depend on $(\hat{\lambda}_+,r)$, all these derivatives are meromorphic functions on the elliptic curve~\eqref{hat elliptic curve}. Inserting \eqref{r and e1 e2 e3} in~\cite[Chapter~13.13 (13)]{Bat} gives $\wp(z_+-\omega')-e_3=\frac{(e_3-e_1)(e_3-e_2)}{\wp(z_+)-e_3}=\frac{(-r^{-1})(-r)}{-\hat{\lambda}_+}=-\hat{\lambda}_+^{-1}$. Now~\cite[Chapter~13.12]{Bat} yields
\begin{align*}
\frac{\partial\Re(\tilde{\tau}_a)}{\partial\hat{\lambda}_+}&
=\frac{\omega'(\hat{\lambda}_++\hat{\lambda}_+^{-1})-2(\omega'e_3+\eta')}{2\pi i\hat{\nu}_+},&\frac{\partial\Im(\tilde{\tau}_a)}{\partial\hat{\lambda}_+}&=\frac{\omega'(\hat{\lambda}_+^{-1}-\hat{\lambda}_+)}{2\pi\hat{\nu}_+}.
\end{align*}
At $\hat{\lambda}_+\in\{-r,-r^{-1}\}$ both functions take values which do not depend on $r$:
\begin{align}\label{values tau}
\Re(\tilde{\tau}_a)|_{\hat{\lambda}_+=-r}&\!=\!-1,&\Re(\tilde{\tau}_a)|_{\hat{\lambda}_+=-r^{-1}}&\!=\!-1,&\Im(\tilde{\tau}_a)|_{\hat{\lambda}_+=-r}&\!=\!0,&\Im(\tilde{\tau}_a)|_{\hat{\lambda}_+=-r^{-1}}&\!=\!0.
\end{align}
Since the corresponding total derivatives with respect to $r$ vanishes, we conclude that
\begin{align*}
\frac{\partial\Re(\tilde{\tau}_a)}{\partial r}&-\frac{\partial\Re(\tilde{\tau}_a)}{\partial \hat{\lambda}_+}&\text{and}&&
\frac{\partial\Im(\tilde{\tau}_a)}{\partial r}&-\frac{\partial\Im(\tilde{\tau}_a)}{\partial \hat{\lambda}_+}&\text{have no pole at }&\hat{\lambda}_+=-r,\\
\frac{\partial\Re(\tilde{\tau}_a)}{\partial r}&+r^{-2}\frac{\partial\Re(\tilde{\tau}_a)}{\partial \hat{\lambda}_+}&\text{and}&&
\frac{\partial\Im(\tilde{\tau}_a)}{\partial r}&+r^{-2}\frac{\partial\Im(\tilde{\tau}_a)}{\partial \hat{\lambda}_+}&\text{have no pole at }&\hat{\lambda}_+=-r^{-1}.
\end{align*}
Since $\big(\Re(\tilde{\tau}_a),\Im(\tilde{\tau}_a)\big)$ have at $\hat{\lambda}_+=0,\infty$ the same simple poles as $\hat{\lambda}_+^{\pm\frac{1}{2}}$, we obtain that
\begin{align*}
\frac{\partial\Re(\tilde{\tau}_a)}{\partial r}&+\frac{2\hat{\lambda}_+}{\omega'}\frac{\partial\omega'}{\partial r}\frac{\partial\Re(\tilde{\tau}_a)}{\partial \hat{\lambda}_+}&\text{and}&&
\frac{\partial\Im(\tilde{\tau}_a)}{\partial r}&+\frac{2\hat{\lambda}_+}{\omega'}\frac{\partial\omega'}{\partial r}\frac{\partial\Im(\tilde{\tau}_a)}{\partial\hat{\lambda}_+}&\text{have no pole at }&\hat{\lambda}_+=0,\\
\frac{\partial\Re(\tilde{\tau}_a)}{\partial r}&-\frac{2\hat{\lambda}_+}{\omega'}\frac{\partial\omega'}{\partial r}\frac{\partial\Re(\tilde{\tau}_a)}{\partial \hat{\lambda}_+}&\text{and}&&
\frac{\partial\Im(\tilde{\tau}_a)}{\partial r}&-\frac{2\hat{\lambda}_+}{\omega'}\frac{\partial\omega'}{\partial r}\frac{\partial\Im(\tilde{\tau}_a)}{\partial\hat{\lambda}_+}&\text{have no pole at }&\hat{\lambda}_+=\infty.
\end{align*}
Both derivatives with respect to $r$ and $\hat{\lambda}_+$ are anti-symmetric with respect to $\sigma$ and uniquely determined by their poles. We insert $r+r^{-1}=-3e_3$ (see \eqref{r and e1 e2 e3}) and calculate
\begin{align*}
\frac{\partial\Re(\tilde{\tau}_a)}{\partial r}&\!=\!\frac{2\frac{\partial\omega'}{\partial r}(\hat{\lambda}_+^2\!-\!1)}{2\pi i\hat{\nu}_+},&\frac{\partial\Im(\tilde{\tau}_a)}{\partial r}&\!=\!\frac{\omega'(r^{-2}\!-\!1)\hat{\lambda}_+\!-\!2\frac{\partial\omega'}{\partial r}\hat{a}(\hat{\lambda}_+)}{2\pi\hat{\nu}_+},&\frac{\partial\omega'}{\partial r}&\!=\!\frac{2\eta'\!-\!\omega'e_3}{2(1-r^2)}.
\end{align*}
By~\cite[Chapter~13.13~(13)]{Bat} the last equation follows also from $\omega'=-\int^0_{\omega'}dz_+
=\int_0^{\infty}\frac{d\hat{\lambda}_+}{2\hat{\nu}_+}$:
\begin{multline*}
\frac{\partial\omega'}{\partial r}=\int_0^{\infty}\frac{\hat{\lambda}_+(r^{-2}-1)d\hat{\lambda}_+}{4\hat{\nu}_+(\hat{\lambda}_++r)(\hat{\lambda}_++r^{-1})}=\int_0^\infty\left(\frac{r^{-2}}{\hat{\lambda}_++r^{-1}}-\frac{1}{\hat{\lambda}_++r}\right)\frac{d\hat{\lambda}_+}{4\hat{\nu}_+}=\\=\int_0^{\omega'}\left(\frac{\wp(z_++\omega+\omega')-e_2}{r(r-r^{-1})}-\frac{\wp(z_+-\omega)-e_1}{r(r^{-1}-r)}\right)\frac{dz_+}{2}=\\=\frac{1}{2(1-r^2)}\left.\left(\zeta(z_+-\omega)+\zeta(z_++\omega+\omega')-e_3z_+\right)\right|_{z_+=0}^{z_+=\omega'}=\frac{2\eta'-\omega'e_3}{2(1-r^2)}.
\end{multline*}
Moreover, the first integral does not vanish for $r\in(0,1)$. In order to calculate the real Jacobian of the restriction of $T$ to $\mathcal{M}_2^2$ we again use the real parameter $\varphi$ instead of $\hat{\lambda}_+=-e^{4i\varphi}$. Then the determinant of the Jacobian can be calculated explicitly as
$$\frac{\partial\Re(\tilde{\tau}_a)}{\partial\varphi}\frac{\partial\Im(\tilde{\tau}_a)}{\partial r}-\frac{\partial\Im(\tilde{\tau}_a)}{\partial\varphi}\frac{\partial\Re(\tilde{\tau}_a)}{\partial r}=
\frac{4((\omega'e_3+\eta')^2-\omega'^2)}{\pi^2(1-r^2)}.$$
The complex continuation of $\Re(\tilde{\tau}_a)$ is purely imaginary on $\hat{\lambda}_+\in\mathbb{R}$ and takes the same values~\eqref{values tau} at $\hat{\lambda}_+\in\{-r,-r^{-1}\}$. Using the transformation with respect to $\rho$ we conclude
$$0=\hspace{-3mm}\int\limits_{\hat{\lambda}_+=-r^{-1}}^{\hat{\lambda}_+=-r}\hspace{-3mm}\frac{\omega'(\hat{\lambda}_++\hat{\lambda}_+^{-1})-2(\omega'e_3+\eta')}{2\pi i\hat{\nu}_+}d\hat{\lambda}_+=2\hspace{-3mm}\int\limits_{\hat{\lambda}_+=-1}^{\hat{\lambda}_+=-r}\hspace{-3mm}\frac{\omega'(\hat{\lambda}_++\hat{\lambda}_+^{-1})-2(\omega'e_3+\eta')}{2\pi i\hat{\nu}_+}d\hat{\lambda}_+.$$
Therefore, $d\Re(\tilde{\tau}_a)$ has a single root in each interval $\hat{\lambda}_+\in(-r^{-1},-1)$ and $(-1,-r^{-1})$. This implies $\frac{\omega'e_3+\eta'}{\omega'}<-1$ and the fact that the former Jacobian has negative determinant as well as the restriction of $T$ to $\mathcal{M}_2^2$ is an immersion.\vspace{3mm}

For $\hat{\lambda}_+=\pm1$ the real part of $\tilde{\tau}_a$ is zero and $\pm 1$, respectively. For fixed $r\in(0,1)$ these values correspond to the maximum and the minimum of $\Im(\tilde{\tau}_a)$. For $r=1,$ we have calculated $\tilde{\tau}_a$ (see Figure~\ref{fig:doubleroot_tau_mapping}). Now the Theorem follows from the behavior of $\zeta$ in the limit $r\to 0$ as described in~\cite[Chapter~13.15~(9)]{Bat}.
\end{proof}
\begin{figure}[h]
		\captionsetup{width=0.7\textwidth}
		\centering
		\fbox{
			\begin{minipage}{6.5cm}
				\centering
				\includegraphics[scale=0.5]{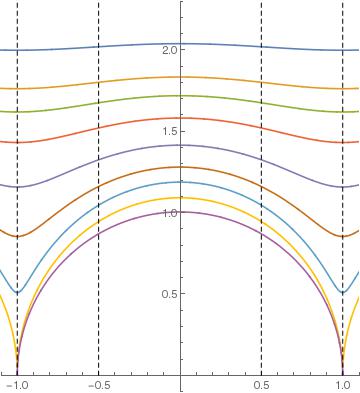}
				\caption{$\tilde{\tau}_a$ for increasing $r$.} 
				\label{fig:real_imag}
			\end{minipage}
		}
\end{figure}
In Figure \ref{fig:real_imag} we plot $\tilde{\tau}_a$ against $t\in(-\omega,\omega)$ for several values $r\in(0,1)$. For increasing $r$ the curves are continuously moving upwards and the described arc will be successively compressed until it resembles a horizontal line.\vspace{3mm}

Finally, we remark that the calculations in the proof of Theorem~\ref{thm immersion}  give explicit formulas for the coefficients of the polynomials $\hat{b}_1$ and $\hat{b}_2$, such that
\begin{align*}
d\ln\mu_1&=\frac{\hat{b}_1(\hat{\lambda})}{2\hat{\nu}}d\ln\hat{\lambda},&
d\ln\mu_2&=\frac{\hat{b}_2(\hat{\lambda})}{2\hat{\nu}}d\ln\hat{\lambda}.
\end{align*}

\section{The Willmore energy of conformal maps $\imm_a:\mathbb{C}/\hat{\Gamma}_a\to\mathbb{H}$}\label{sec willmore}
In this section we use quaternionic function theory~\cite{PP} and define for each $a\in\mathcal{M}_2^1\cup\mathcal{M}_2^2\cup\mathcal{M}_2^3$ a conformal map $\imm_a:\mathbb{C}/\hat{\Gamma}_a\to\mathbb{H}$ for some sublattice $\hat{\Gamma}_a\subset\tilde{\Gamma}_a$. More precisely, we define a quaternionic holomorphic line bundle on $\mathbb{C}/\hat{\Gamma}_a$ with two holomorphic sections $s_1$ and $s_2$. They define the map $\imm_a$ such that $s_2 = s_1 \imm_a$. This map is conformal.\vspace{3mm}

For any $p=(\lambda,\nu)\in\Sigma^\ast$~\eqref{eq:spectral curve} each $\zeta_0\in I(a)$ has at $\lambda$ a nontrivial eigenspace with eigenvalue $\nu$. This eigenspace is away from the roots of $a$ one-dimensional. At simple roots of $a$ $\nu$ vanishes, $\zeta_0$ is nilpotent and the eigenspace is also one-dimensional. We define an eigenfunction $\psi$ of $\zeta$ for $p=(\lambda,\nu)\in\Sigma^\ast$ in terms of the fundamental solution~\eqref{ODE_frame}:
\begin{align}\psi(z)&=F|_{\lambda}^{-1}(z)\chi&\mbox{with}&&z\in\mathbb{C}&\mbox{ and }\chi\in\mathbb{C}^2\setminus\{(0,0)\}&\mbox{such that}&&\zeta_0|_{\lambda}\chi&=\nu\chi.\label{baker akhiezer}\end{align}
Let $\qj=(\begin{smallmatrix}0&1\\-1&0\end{smallmatrix})$ represent the quaternion $\qj$ as a $2\times2$-matrix. The whole set $\mathbb{H}$ is represented by the $2\times2$-matrices $A$ obeying $\qj A=\bar{A}\qj$ or equivalently by the matrices of the form $(\chi,-j\bar{\chi})$ with $\chi\in\mathbb{C}^2$. The involution $\eta=\sigma\circ\rho$ acts on $\zeta_0$ and $F$ as
\begin{align*}
\eta^\ast\zeta_0&=-\qj\bar{\zeta}_0\qj,&
\eta^\ast F(z)&=-\qj\bar{F}(z)\qj.
\end{align*}
Hence $-\qj\chi$ is an eigenvector of $\zeta_0|_{\eta(\lambda)}$ with eigenvalue $\eta(\nu)$, if $\chi$ is an eigenvector of $\zeta_0|_{\lambda}$ with eigenvalues $\nu$, and $F|_{\eta(\lambda)}^{-1}(z)(-\qj\bar{\chi})=-\qj\bar{\psi}(z)$. For two points $p_1$, $p_2$ in $\Sigma^\ast$ and non-trivial eigenvectors $\chi_1,\chi_2\in\mathbb{C}^2\setminus\{(0,0)\}$ of $\zeta_0$ at these two points, the vectors
\begin{align*}
\chi_3&=-\qj\bar{\chi}_1&\mbox{and}&&\chi_4&=-\qj\bar{\chi}_2
\end{align*}
are eigenvectors of $\zeta_0$ at the points $p_3=\eta(p_1)$ and $p_4=\eta(p_2)$. The corresponding functions $\psi_1,\ldots,\psi_4$~\eqref{baker akhiezer} define the following maps $s_1$, $s_2$, $\imm_a:\mathbb{C}\to\mathbb{H}$:
\begin{align}\label{eq:imm}
s_1&=(\psi_1,\psi_3)=(\psi_1,-\qj\bar{\psi}_1),&
s_2&=(\psi_2,\psi_4)=(\psi_2,-\qj\bar{\psi}_2),&
\imm_a&=s_1^{-1}s_2.
\end{align}
We remark that $s_1$ has no roots, since $F$~\eqref{ODE_frame} is for all $(\lambda,z)\in\mathbb{C}^\times\times\mathbb{C}$ an invertible $2\times 2$ matrix. We want to consider $s_1$ and $s_2$ as sections of a quaternionic line bundle. Since the monodromies along two generators $\hat{\omega}_1$ and $\hat{\omega}_2$ of $\hat{\Gamma}_a$ commute with $\zeta_0$, the translations by $\hat{\omega}_j$ with $j=1,2$ act on $s_1$ and $s_2$ as left multiplication by the eigenvalues $\hat{\mu}_j$ of $M_{\hat{\omega}_j}$
\begin{align*}
\left(\begin{smallmatrix}\hat{\mu}_j(p_1)&0\\0&\hat{\mu}_j(p_3)\end{smallmatrix}\right)&=
\left(\begin{smallmatrix}\hat{\mu}_j(p_1)&0\\0&\bar{\hat{\mu}}_j(p_1)\end{smallmatrix}\right),&
\left(\begin{smallmatrix}\hat{\mu}_j(p_2)&0\\0&\hat{\mu}_j(p_4)\end{smallmatrix}\right)&=
\left(\begin{smallmatrix}\hat{\mu}_j(p_2)&0\\0&\bar{\hat{\mu}}_j(p_2)\end{smallmatrix}\right)
\end{align*}
respectively. They commute with the right action of $\mathbb{H}$ on the trivial bundle $\mathbb{C}\times\mathbb{H}$, if they are all real. In this case the representation $\hat{\Gamma}_a\to\mathbb{C}^\times$ with $\hat{\omega}_j\mapsto\hat{\mu}_j(p_i)$, induces for each $i=1,2$ on $\mathbb{C}/\hat{\Gamma}_a$, a quaternionic line bundle with section $(\psi_i,-\qj\psi_i)$. Furthermore, both quaternionic line bundles coincide, if the following closing condition~(C) holds. In this case a direct calculation confirms that the function $\imm_a$ is periodic with respect to $\hat{\Gamma}_a$.
\begin{enumerate}
\item[(C)] There exist $k_1, k_2 \in \mathbb{R}\setminus\{0\}$ such that $\hat{\mu}_j(p_i) = k_j$ for $j=1,2$ and $i=1,2$.
\end{enumerate}
Due to Theorem~\ref{thm orbits} the fundamental solutions $F,\tilde{F}$ of two $\zeta_0,\tilde{\zeta}_0\in I(a)$ differ by a translation with respect to $z$ and a left multiplication by an invertible $2\times 2$-matrix. The corresponding sections $s_1,\tilde{s}_1$ and $s_2,\tilde{s}_2$~\eqref{eq:imm}
differ by translation and right multiplication by quaternions, respectively. Finally, $\imm_a$ and $\tilde{\imm}_a$ differ by translation with respect to $z\in\mathbb{C}$ and by left and right multiplication with quaternions. The left and right multiplication by the quaternions are global conformal maps of $\mathbb{H}$. This shows that $\imm_a$ and $\tilde{\imm}_a$ induce up to translations with respect to $z$ and global conformal maps of $\mathbb{H}$ the same maps to $\mathbb{H}$.\vspace{3mm}

We remark that cmc tori in $\mathbb{S}^3$ are special cases of our construction. If $p_1$ is a fixed point of $\rho$, then the map $s_1^{-1}\qi s_1$ to the unit imaginary quaternions (which are isomorphic to the 2-sphere in the imaginary quaternions $\simeq\mathbb{R}^3$) is harmonic. If $p_2$ is a fixed point of $\rho$ then $s_2^{-1}\qi s_2$ is harmonic. If both conditions hold, then the Cartesian product $s_1^{-1}\qi s_1\times s_2^{-1}\qi s_2$ is harmonic and $\imm_a$ is a cmc torus in unit quaternions $\simeq\mathbb{S}^3$ (see~\cite{KSS}). In quaternionic function theory these two maps are called left and right normal~\cite[Section~6.2]{Bo}.
\begin{lemma}\label{lemma:g1}
For $a \in \mathcal{M}_2^2 \cup \mathcal{M}_2^3$ we identify $\bar{\Sigma}$ with $\mathbb{C}/(2\omega\mathbb{Z} + 2\omega'\mathbb{Z})$. Let $(p_1,p_2,p_3,p_4)$ correspond to $(z_+,\omega,-\bar{z}_++\omega',\omega+\omega')$. Then the following equations hold:
\begin{align*}
p_3&=\eta(p_1),&p_4&=\eta(p_2),&(\mu_1,\mu_2)|_{p_1, p_3} &= (-1,1),&(\mu_1,\mu_2)|_{p_2, p_4} &= (1,-1).
\end{align*}
\end{lemma}
\begin{proof}
In Section \ref{Sec:g1} we have seen 
\begin{align*}
\ln\mu_1|_{p_1}&=i\pi,&\ln\mu_1|_{p_2}&=0,&\ln\mu_2|_{p_1}&=0,&\ln\mu_2|_{p_2}&=i\pi.
\end{align*}
The transformations~\eqref{eq:involutions z} and~\eqref{involutions mu} imply the four equations of the Lemma.
\end{proof}

\begin{corollary}\label{Corollary:g2}
For $a \in \mathcal{M}_2^1$ we can label the roots of $a$ by $p_1, p_2, p_3, p_4$ in such a way that the statements from Lemma \ref{lemma:g1} still hold true.
\end{corollary}

\begin{proof}
Choose a path from $a \in \mathcal{M}_2^1$ to $\mathcal{M}_2^2$. Due to Theorem~\ref{continuous T}, the lattice $\tilde{\Gamma}_a$ depends continuously on $a$. The generators $\omega_1, \omega_2$ at the end point of the path extend uniquely to generators of $\tilde{\Gamma}_a$ along the path. The eigenvalues $\mu_1, \mu_2$ of the corresponding monodromies also extend along the path. Since the roots of $a$ are fixed points of $\sigma$, $\mu_1$ and $\mu_2$ take values $\pm 1$ there. These values are constant along the path.
\end{proof}
We shall now construct a sublattice $\hat{\Gamma}_a\subset\tilde{\Gamma}_a$ such that the condition (C) is satisfied.
\begin{lemma}\label{lemma_new_lattice}
Consider the lattice $\tilde{\Gamma}_a$ with generators $\omega_1, \omega_2$ constructed in Section \ref{Sec:g1} for $a\in\mathcal{M}_2^2\cup\mathcal{M}_2^3$ and in the proof of Corollary \ref{Corollary:g2} for $a \in \mathcal{M}_2^1$. The eigenvalues of the sublattice $\hat{\Gamma}_a$ generated by $\hat{\omega}_1=\omega_1+\omega_2$ and $\hat{\omega}_2=\omega_2-\omega_1$ satisfy $(\hat{\mu}_1,\hat{\mu}_2)|_{p_1,p_2,p_3,p_4} = (-1,-1)$. The transformation $\tilde{\Gamma}_a \mapsto \hat{\Gamma}_a$ induces a transformation $\tilde{\tau}_a \mapsto \hat{\tau}_a$ given by
$$
\hat{\tau}_a = \frac{\tilde{\tau}_a-1}{\tilde{\tau}_a+1}.
$$
\end{lemma}

\begin{proof}
We get for a sublattice $\hat{\Gamma}_a \subset \tilde{\Gamma}_a$
\begin{align*}
\hat{\mu}_1&=\mu_1\mu_2,&\hat{\mu}_2&=\mu_1\mu_2^{-1}.
\end{align*}
With $\tilde{\tau}_a=\frac{\omega_2}{\omega_1}$ we obtain $\hat{\tau}_a=\frac{\hat{\omega}_2}{\hat{\omega}_1}=\frac{\tilde{\tau}_a-1}{\tilde{\tau}_a+1}$ and the claim is proved.
\end{proof}
\begin{remark}
The volume of $\mathbb{C}/\hat{\Gamma}_a$ is $\vol(\mathbb{C}/\hat{\Gamma}_a) =|\tilde{\Gamma}_a/\hat{\Gamma}_a|\cdot\vol(\mathbb{C}/\tilde{\Gamma}_a)=2\cdot\vol(\mathbb{C}/\tilde{\Gamma}_a)$.
\end{remark}
\begin{remark}
We want to emphasize that the map 
\begin{align*}
\hat{T}:\,\mathcal{M}_2^2&\to\mathcal{F},&a&\mapsto \hat{\tau}_a
\end{align*}
is continuous but not differentiable at the points corresponding to $(r,z_+) \in (0,1) \times \{\omega\}$. In fact, as a function from the fundamental domain $\mathcal{F}$ onto itself $\tilde{\tau}_a\mapsto\hat{\tau}_a$ is given by
\begin{align*}
\hat{\tau}_a&=
\begin{cases}
\frac{\tilde{\tau}_a-1}{\tilde{\tau}_a+1}&\mbox{for }\Re(\tilde{\tau}_a)<0 \\
\frac{1+\tilde{\tau}_a}{1-\tilde{\tau}_a}&\mbox{for }\Re(\tilde{\tau}_a)\geq0
\end{cases}.
\end{align*}
\end{remark}
\begin{theorem}
For all $a \in \mathcal{M}_2^1 \cup \mathcal{M}_2^2 \cup \mathcal{M}_2^3$ the map $\imm_a:\mathbb{C}/\hat{\Gamma}_a\to\mathbb{H}$ is a conformal immersion.
\end{theorem}

\begin{proof}
Let us first verify condition~(C). From Lemma \ref{lemma_new_lattice} we set $k_1 = k_2 = -1 \in \mathbb{R}$.

Next we show that $\imm_a$ has the form of the example in~\cite[p.\ 395]{PP}. Due to~\eqref{eq:commute} the 1-form 
\begin{align*}
\Omega&=U(\zeta)dx+V(\zeta))dy=\begin{pmatrix}\frac{\alpha}{2}dz-\frac{\bar{\alpha}}{2}d\bar{z}&-\gamma d\bar{z}-\gamma^{-1}\lambda^{-1}dz\\\gamma dz+\gamma^{-1}\lambda d\bar{z}&\frac{\bar{\alpha}}{2}d\bar{z}-\frac{\alpha}{2}dz\end{pmatrix}
\end{align*}
defines on the trivial quaternionic line bundle over $x+iy\in\mathbb{R}^2$ a real flat connection depending on $\lambda$. The left multiplication with $\qi=(\begin{smallmatrix}i&0\\0&-i\end{smallmatrix})$ endows this bundle with a complex structure in the sense of quaternionic function theory~\cite{PP}. The $(0,1)$-part
$$\tfrac{1}{2}(\Omega+\qi\ast\Omega)=\begin{pmatrix}-\frac{\bar{\alpha}}{2}d\bar{z}&-\gamma d\bar{z}\\\gamma dz&-\frac{\alpha}{2}dz\end{pmatrix}$$
(with $\ast dz=idz$ and $\ast d\bar{z}=-id\bar{z}$~\cite{PP}) does not depend on $\lambda$ and defines an anti-holomorphic structure in the sense of \cite{PP}. Both factors $s_1$ and $s_2$~\eqref{eq:imm} are holomorphic sections of this quaternionic line bundle and $\imm_a$ is a conformal map.
\end{proof}
For $a=(\lambda^2+1)^2$ the immersion $\imm_a$ is the Clifford torus. The Willmore energy of $\imm_a$ is
$$W = \int_{\mathbb{C}/\hat{\Gamma}_a} H^2 \,dA.$$
Here, $H$ denotes the mean curvature and $dA$ the induced volume form. By construction of $\imm_a$ the left normal $N$~\cite{PP} and the Hopf field $Q=\frac{1}{4}(NdN+\ast dN)$ are equal to
\begin{align}\label{left normal}
N&=s_1^{-1}\qi s_1,&\tfrac{1}{4}(NdN+\ast dN)&=s_1^{-1}\left(\begin{smallmatrix}0&-\gamma dz\\\gamma d\bar{z}&0\end{smallmatrix}\right)s_1.
\end{align}
Recall that due to \cite[p. 395]{PP}, the Willmore energy is the $L^2$-norm of the Hopf field $Q$:
$$
W = \int_{\mathbb{C}/\hat{\Gamma}_a}(H^2-K-K^\perp)\,dA = \int_{\mathbb{C}/\hat{\Gamma}_a} 4|Q|^2 = \int_{\mathbb{C}/\hat{\Gamma}_a} 4\gamma^2 \,dx \wedge dy.$$
Note that the topological constant $\int_{\mathbb{C}/\hat{\Gamma}_a}(K+K^\perp)dA$ vanishes in the present situation, since it vanishes for the Clifford torus.
\begin{theorem}
For all $a\in\mathcal{M}_2^1 \cup \mathcal{M}_2^2 \cup \mathcal{M}_2^3$ the Willmore energy $W(a)$ of $\imm_a$ is equal to
\begin{align*}
W(a) = \int\limits_{\mathbb{C}/\hat{\Gamma}_a}4 \gamma^2 \,dx \wedge dy = \int\limits_{\mathbb{C}/\tilde{\Gamma}_a}8 \gamma^2 \,dx \wedge dy = 4i\Res\limits_{\lambda=0}\ln (\mu_2)d\ln (\mu_1).
\end{align*}
For $a\in\mathcal{M}_2^2\cup\mathcal{M}_2^3$ the immersion $\imm_a$ is constrained Willmore with Willmore energy
$$W(a)=\frac{16\pi(\omega'e_3+\eta')\wp'(z_+)}{e_3-\wp(z_+)}.$$
\end{theorem}

\begin{proof}
From~\eqref{ODE_frame} we obtain $d\psi=-\Omega\psi$. 
We define $\kappa_1=\frac{i\bar{\omega}_2}{\Im(\omega_2\bar{\omega}_1)}$ and $\kappa_2=\frac{i\bar{\omega}_1}{\Im(\omega_1\bar{\omega}_2)}$ with
\begin{align*}
\Re(\kappa_1\omega_1)&=1=\Re(\kappa_2\omega_2),&\Re(\kappa_1\omega_2)&=0=\Re(\kappa_2\omega_1).
\end{align*}
In a neighborhood $p = (\lambda, \nu)$ of $p_0 = (0, 0)$ the following function $\tilde{\psi}$ is by definition of $\kappa_1$ and $\kappa_2$ double periodic with respect to $\tilde{\Gamma}_a$ and obeys due to $d\ln\gamma=-\alpha dz-\bar{\alpha}d\bar{z}$~\eqref{LaxZeta}:
\begin{gather*}
\tilde{\psi}(z)=\left(\begin{smallmatrix}\tilde{\psi}_1(z)\\\tilde{\psi}_2(z)\end{smallmatrix}\right)=\exp(\ln\mu_1\Re(\kappa_1z)+\ln\mu_2\Re(\kappa_2z))
\left(\begin{smallmatrix}\gamma^{\frac{1}{2}}(z)&0\\0&\gamma^{-\frac{1}{2}}(z)\nu^{-1}\end{smallmatrix}\right)\psi(z),\\
d\tilde{\psi}=\begin{pmatrix}\ln\mu_1\Re(\kappa_1dz)+\ln\mu_2\Re(\kappa_2dz)-\alpha dz&\lambda^{-1}\nu dz+\nu\gamma^2d\bar{z}\\-\nu^{-1}dz-\gamma^{-2}\lambda\nu^{-1}d\bar{z}&\ln\mu_1\Re(\kappa_1dz)+\ln\mu_2\Re(\kappa_2dz)+\alpha dz\end{pmatrix}\tilde{\psi}.
\end{gather*}
We utilize Lemma~\ref{1-forms} and expand the following functions nearby $p_0$ (see \cite[Section~3]{Hi}):
\begin{gather*}
\begin{aligned}
\ln\mu_1&=-\omega_1\nu^{-1}+W_1\nu+\mathbf{O}(\nu^3),&
\ln\mu_2&=-\omega_2\nu^{-1}+W_2\nu+\mathbf{O}(\nu^3),& W_1, W_2 \in \mathbb{C},
\end{aligned}\\
\ln\mu_1\Re(\kappa_1dz)+\ln\mu_2\Re(\kappa_2dz)=-dz\nu^{-1}+\left(\tfrac{(W_1\bar{\omega}_2-W_2\bar{\omega}_1)}{2\Im(\omega_2\bar{\omega}_1)}dz\!+\!\tfrac{(W_2\omega_1-W_1\omega_2)}{2\Im(\omega_2\bar{\omega}_1)}d\bar{z}\right)i\nu+\mathbf{O}(\nu^3).\end{gather*}
Therefore we may expand the $dz$ and $d\bar{z}$ parts of $d\tilde{\psi}$ independently:
$$d\tilde{\psi}=\left(-\!\begin{pmatrix}1&1\\1&1\end{pmatrix}\nu^{-1}dz+
\mathbf{O}(\nu^0)dz+\begin{pmatrix}i\tfrac{W_2\omega_1-W_1\omega_2}{2\Im(\omega_2\bar{\omega}_1)}&\gamma^2\\\gamma^{-2}&i\tfrac{W_2\omega_1-W_1\omega_2}{2\Im(\omega_2\bar{\omega}_1)}\end{pmatrix}\nu d\bar{z}+\mathbf{O}(\nu^2)d\bar{z}\right)\tilde{\psi}.$$
We normalize the vector $\chi$ in~\eqref{baker akhiezer} by $\chi_1=\gamma^{-\frac{1}{2}}(0)$ such that $\tilde{\psi}_1|_{z=0}=1$. Since $\chi$ is an eigenvector of $\zeta_0$ with eigenvalue $\nu$ this implies $\chi_2=-\nu\gamma^{\frac{1}{2}}(0)+\mathbf{O}(\nu^2)$ and $\tilde{\psi}_2|_{z=0}=-1+\mathbf{O}(\nu)$. Thus $\tilde{\psi}$ takes at $p_0$ the value $\left(\begin{smallmatrix}1\\-1\end{smallmatrix}\right)$ in the kernel of the leading matrix $\left(\begin{smallmatrix}1&1\\1&1\end{smallmatrix}\right)$. Since the $d\bar{z}$-part vanishes at $p_0$ the double periodic function $\tilde{\psi}$ is constant at $p_0$:
\begin{align*}
\tilde{\psi}_1&=\tilde{\psi}_{1,0}+\nu\tilde{\psi}_{1,1}+\mathbf{O}(\nu^2),&
\tilde{\psi}_2&=\tilde{\psi}_{2,0}+\nu\tilde{\psi}_{2,1}+\mathbf{O}(\nu^2) \\
\bar{\partial}\tilde{\psi}_{1,0}&=0=\bar{\partial}\tilde{\psi}_{2,0}&\mbox{i.e.\ }&\tilde{\psi}_{1,0}\mbox{ and }\tilde{\psi}_{2,0}\mbox{ are constant.}
\end{align*}
Now we obtain from the previous expansion
\begin{align*}
\bar{\partial}\tilde{\psi}_{1,1}&=i\tfrac{W_2\omega_1-W_1\omega_2}{2\Im(\omega_2\bar{\omega}_1)}-\gamma^2,&\bar{\partial}\tilde{\psi}_{2,1}&=\gamma^{-2}+i\tfrac{W_1\omega_2-W_2\omega_1}{2\Im(\omega_2\bar{\omega}_1)}.
\end{align*}
Since $\Im(\omega_2\bar{\omega}_1)=\frac{1}{\Im(\kappa_2\bar{\kappa}_1)}>0$ is $\vol(\mathbb{C}/\tilde{\Gamma}_a)$, the integration over $\mathbb{C}/\tilde{\Gamma}_a$ yields
$$4i\Res_{\lambda=0}\ln\mu_2d\ln\mu_1=4i(W_2\omega_1-W_1\omega_2)=\int\limits_{\mathbb{C}/\tilde{\Gamma}_a}8i\tfrac{W_2\omega_1-W_1\omega_2}{2\Im(\omega_2\bar{\omega}_1)}dx\wedge dy=\int\limits_{\mathbb{C}/\tilde{\Gamma}_a}8\gamma^2dx\wedge dy.$$
For $a\in\mathcal{M}_2^2\cup\mathcal{M}_2^3$ the function $\lambda$ takes at both points $p_1$ and $p_2$ the same value. Therefore the left normal $N$~\eqref{left normal} is defined in terms of the frame~$F$~\eqref{ODE_frame} of the solution $u=\ln\gamma$ of the $\sinh$-Gordon equation. In particular, $N$ defines a harmonic map to $\mathbb{S}^2$~\cite{Hi}. Now \cite[Theorem~6.1 or Lemma~6.3]{Bo} implies that $\imm_a$ is constrained Willmore. Due to~\eqref{involutions mu} the residue of $\ln\mu_2d\ln\mu_2$ at $\lambda=\infty$ is the complex conjugated and therefore the negative of the residue at $\lambda=0$. In this case $\Im(\omega_2\bar{\omega}_1)=\frac{|\omega|^2}{\Im(\omega'/\omega)}>0$, and~\eqref{lnmu1}-\eqref{lnmu2} yields
\begin{align*}
W(a)&=4\pi\Res_{z=0}(\omega'(\zeta(z)+\zeta(z-\omega')+\eta')-2\eta'z)
\frac{\wp(z-\omega')-\wp(z)}{\zeta(z_+)-\zeta(z_+-\omega')-\eta'}dz\\
&=\frac{8\pi(\omega'e_3+\eta')}{\zeta(z_+)-\zeta(z_+-\omega')-\eta'}
=\frac{16\pi(\omega'e_3+\eta')\wp'(z_+)}{e_3-\wp(z_+)}.
\end{align*}
In the last equation we used~\cite[Chapter~13.13~(18)]{Bat}.\end{proof}
The sub-family $(\imm_a)_{a\in\mathcal{M}_2^3}$ are rotational cmc tori in the unit quaternions $\simeq\mathbb{S}^3$ with rectangular conformal classes. They are isothermic surfaces. In terms of the coordinates $r=1$ and $z_+=\frac{\pi i}{4}+t$ of~\eqref{genus0} their Willmore functional is
$$W(a)=2\pi^2(2\cosh^2(t)-1).$$
The minimum $W(a)=2\pi^2$ is taken for $t=0$, which corresponds to the Clifford torus. For non-rectangular conformal classes they are no isothermic surfaces. In Figure~\ref{fig:Willmore_tau} we plot the Willmore functional $W(a)/(4\pi)$ of the family $\imm_a$ of constrained Willmore tori in dependence of the conformal class $\hat{\tau}_a$. We conjecture that this family represents the minimum of the Willmore functional for all conformal classes nearby $\hat{\tau}_a=i$. In a subsequent work we shall construct a second family of constrained Willmore tori in terms of the solutions of the $\sinh$-Gordon equation presented here. The spectral genus of the second family varies between zero and two. Nearby the Clifford torus it coincides with the family $\imm_a$. Only for spectral genus two and on some parts for spectral genus one it differs form the family presented here and has less Willmore energy.

\begin{figure}[h]
		\captionsetup{width=0.7\textwidth}
		\centering
		\fbox{
			\begin{minipage}{8.0cm}
				\centering
				\includegraphics[scale=0.47]{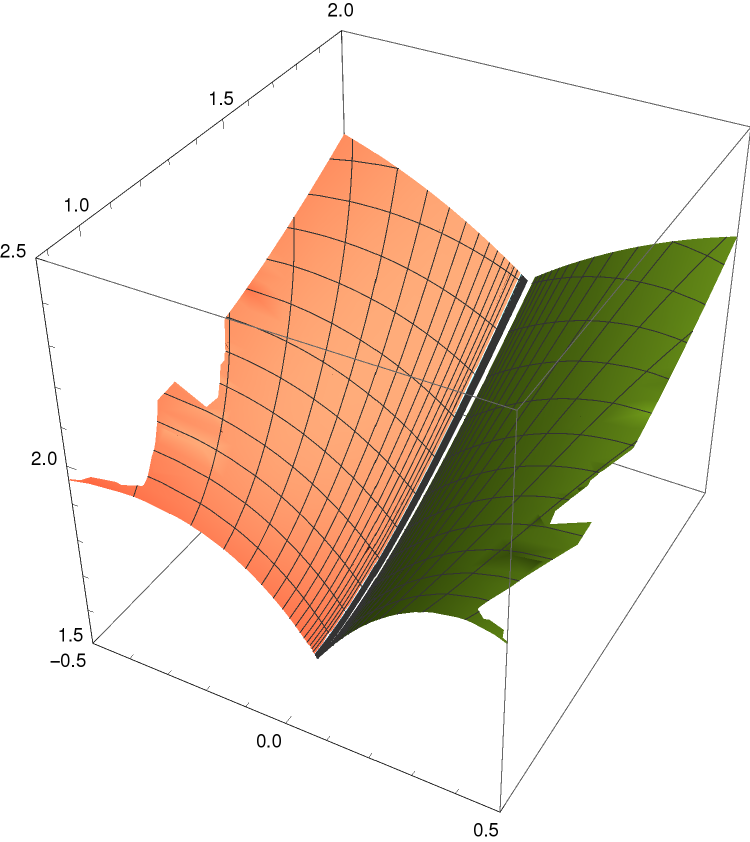}
				\caption{$W(a)/(4\pi)$ in dependence of $\hat{\tau}_a$.} 
				\label{fig:Willmore_tau}
			\end{minipage}
		}
\end{figure}

\thispagestyle{fancy}	
\fancyhf{}
\renewcommand{\sectionmark}[1]{\markboth{#1}{}} 
 \fancyhead[EL]{\scshape References}
 \fancyhead[OR]{\scshape References}
 \fancyfoot[EL,OR]{\thepage}
 \renewcommand{\headrulewidth}{0.5pt}

\end{document}